\documentclass[11pt]{amsart}
\usepackage{setspace}
\usepackage[utf8]{inputenc}
\usepackage[english]{babel}
\usepackage{amsmath, amsthm,amsfonts,amssymb}
\usepackage{dsfont}
\usepackage{enumerate}
\usepackage{graphicx}
\usepackage{subfig}
\usepackage[left=2.5cm,right=2.5cm,top=2.2cm,bottom=2.2cm, includefoot]{geometry}
\usepackage{float}
\usepackage{color}

\newcommand{\R}{\mathbb{R}} 
\newcommand{\C}{\mathbb{C}} 
\newcommand{\N}{\mathbb{N}} 
\newcommand{\Z}{\mathbb{Z}} 

\newcommand{\Sp}{\mathbb{S}} 
\newcommand{\SPN}{\Sp^{N-1}}

\newcommand{\cC}{{\mathcal C}}

\newcommand{\cF}{{\mathcal F}}

\newcommand{\cR}{{\mathcal R}}   
\newcommand{\cS}{{\mathcal S}}

\newcommand{\wdh}{\widehat}
\newcommand{\wdc}{\check}
\newcommand{\intl}{\int\limits}

\newcommand{\eps}{\varepsilon}

\newcommand{\RN}{\R^{N}}
\newcommand{\ri}{\text{i}}

\newcommand{\ceinftyc}{\mathcal{C}^{\infty}_{c}}   
\newcommand{\SR}{\mathcal{S}}                      

\newcommand{\sQ}{\text{\tiny $Q$}}

\DeclareMathOperator*{\id}{id}

\newtheorem{defi}{Definition}[section]
\newtheorem{theorem}{Theorem}[section] 
\newtheorem{remark}[theorem]{Remark}
\newtheorem{lemma}[theorem]{Lemma}
\newtheorem{proposition}[theorem]{Proposition}
\newtheorem{corr}[theorem]{Corollary}

\numberwithin{equation}{section}

\title[Fourier-extension estimates for symmetric functions]{Fourier-extension estimates for symmetric functions and applications to nonlinear Helmholtz equations}
\author{Tobias Weth and Tolga Ye\c{s}il} 
\date{}

\thanks{Goethe-Universit\"{a}t Frankfurt, Institut f\"{u}r Mathematik, Robert-Mayer-Str. 10,
   D-60629 Frankfurt, Germany\\
Email addresses: weth@math.uni-frankfurt.de, yesil@math.uni-frankfurt.de}

 \address{Goethe-Universit\"{a}t Frankfurt, Institut f\"{u}r Mathematik, Robert-Mayer-Str. 10, D-60629 Frankfurt, Germany }
 \email{weth@math.uni-frankfurt.de, yesil@math.uni-frankfurt.de}

\begin{document}
\maketitle

\begin{abstract}
  We establish weighted $L^p$-Fourier-extension estimates for $O(N-k) \times O(k)$-invariant functions defined on the unit sphere $\Sp^{N-1}$, allowing for exponents $p$ below the Stein-Tomas critical exponent $\frac{2(N+1)}{N-1}$. Moreover, in the more general setting of an arbitrary closed subgroup $G \subset O(N)$ and $G$-invariant functions, we study the implications of weighted Fourier-extension estimates with regard to boundedness and nonvanishing properties of the corresponding weighted Helmholtz resolvent operator. Finally, we use these properties to derive new existence results for $G$-invariant solutions to the nonlinear Helmholtz equation
  $$
- \Delta u - u = Q(x)|u|^{p-2}u, \quad u \in W^{2,p}(\RN),  
$$
where $Q$ is a nonnegative bounded and $G$-invariant weight function. 
\end{abstract}

\section{Introduction}
Starting with the pioneering work of Stein (cf. \cite{Stein}), Tomas \cite{Tomas1975} and Strichartz \cite{Strichartz}, Fourier restriction and extension estimates have been receiving extensive attention due to their various applications, especially to partial differential equations. For an overview on classical results and recent progress, we refer the reader to e.g. \cite{Foschi2017,Stein-K,Tao2004}. In its classical form, the famous Fourier extension theorem of Stein and Tomas (see e.g. \cite[\S8: Corollary 5.4]{Stein-K}) states that 
the inverse Fourier transform $\wdc{F}_{\sigma}$ of $F \in L^{2}(\Sp^{N-1})$, given by
\begin{equation*}\label{eqn:Intro}
\wdc{F}_{\sigma} (x)  =(2\pi)^{-\frac{N}{2}}\int\limits_{\Sp^{N-1}} e^{i\omega\cdot x}F(\omega)~d\sigma(\omega) 
\end{equation*}
belongs to $L^{q}(\RN)$ for $N \geq 2$ if $q \geq \frac{2(N+1)}{N-1}$, and that
\begin{equation}\label{eqn:intro}
\| \wdc{F}_{\sigma} \|_{L^{q}(\RN)} \leq C \left\|F\right\|_{L^{2}(\Sp^{N-1})}
\end{equation}
with a constant $C>0$ depending only on $q$ and $N$. Here $\Sp^{N-1}$ denotes the $(N-1)-$dimensional sphere in $\RN$ and $d\sigma$ the induced Lebesgue measure on  $\Sp^{N-1}$. Due to the Knapp example given by a characteristic function of a small spherical cap in $\Sp^{N-1}$, this range of exponents is known to be sharp for arbitrary functions, see e.g \cite[Chapter 4]{Tao2004}. On the other hand, it is a natural question whether the range of exponents can be improved both by considering weighted $L^q$-norms and by restricting to functions having additional symmetries. A well known and classical observation in this context 
yields that case \eqref{eqn:intro} holds for $q > \frac{2N}{N-1}$ and {\em radial} (and thus constant) functions $F \in L^{2}(\Sp^{N-1}) $, see e.g. \cite[\S8: Proposition 5.1]{Stein-K}.

In the present paper, we analyze this question for more general symmetries with respect to closed subgroups of $O(N)$.

For this we introduce the following definition. 

\begin{defi}
  \label{sec:introduction-definition}
Let $q \ge 1$, let $G \subset O(N)$ be  closed subgroup, and let $Q: \R^N \to \C$ be a measurable function. We call $(G,q,Q)$ an admissible extension triple if there exists a constant $C>0$ with 
\begin{equation}\label{eqn:intro-defi}
\|Q \wdc{F}_{\sigma} \|_{L^{q}(\RN)} \leq C \left\|F\right\|_{L^{2}(\Sp^{N-1})} \qquad \text{for every $G$-invariant function $F \in L^2(\Sp^{N-1})$.}
\end{equation} 
\end{defi}
Here and in the following, a function $F \in L^2(\Sp^{N-1})$ is called $G$-invariant if $F(A \theta)= F(\theta)$ for every $\theta \in \Sp^{N-1}$, $A \in G$. By the remarks above, $(\{\id\},q,1)$ is an admissible extension triple if $q \ge \frac{2(N+1)}{N-1}$ and 
$(O(N),q,1)$ is an admissible extension triple if $q> \frac{2N}{N-1}$. As a further specific example, we mention the subgroup $O(N-1) \cong O(N-1) \times \{\id_\R\} \subset O(N)$ which corresponds to axial symmetry with respect to a fixed axis in $\R^N$. Since a characteristic function of a small spherical cap in $\Sp^{N-1}$ -- as considered in Knapp's example mentioned above -- is axially symmetric, the range for $q$ with $(O(N-1),q,1)$ being an admissible extension triple cannot be extended beyond the value $\frac{2(N+1)}{N-1}$.

If, on the other hand, we consider weight functions $Q \in L^s(\R^N)$ for suitable $s < \infty$, then the range of exponents giving rise to admissible extension triples can be readily extended by applying Hölder's inequality to the LHS of (\ref{eqn:intro-defi}). In particular, this yields that $(\{\id\},q,Q)$ is an 
admissible extension triple if $Q \in L^s(\R^N)$ for some $s \in [1,\infty)$ and $q \ge \max\bigl\{\frac{2s(N+1)}{2(N+1)+s(N-1)},1\bigr\}$. Moreover, $(O(N),q,Q)$ is an admissible extension triple if $Q \in L^s(\R^N)$ for some $s \in [1,\infty)$ and $q \ge \max\bigl\{\frac{2sN}{2N+s(N-1)},1\}$.

%

In the present paper, we are interested in weight functions $Q \in L^\infty(\R^N)$, where Hölder's inequality does not yield an extended range of admissible exponents. The main aims of the paper are the following. First, we wish to detect a class of admissible extension triples corresponding to nontrivial subgroups of $O(N)$ and corresponding to functions $Q \in L^\infty(\R^N)$ which are not $s$-integrable for any $s < \infty$. Second, starting from a range of admissible extension triples $(G,q,Q)$, we wish to derive selfdual $(L^{p'},L^p)$-estimates for the restriction of mappings of the form 
$$
f \mapsto \cR_\sQ f := Q \cR (Q f)
$$
to $G$-invariant functions in the Schwartz space $\cS$ of rapidly decreasing functions in $\R^N$. Here $\cR$ denotes the standard Helmholtz resolvent defined by $\cR f = \Phi \ast f$, where 
\begin{equation}\label{fundsol}
\Phi(x):= \frac{\ri}{4}(2\pi|x|)^{\frac{2-N}{2}}H^{(1)}_{\frac{N-2}{2}}(|x|), \quad\text{for }x\in\RN\backslash\{0\},
\end{equation}
is the fundamental solution of the Helmholtz operator associated with Sommerfeld's 
outgoing radiation condition $\partial_ru(x)-\ri u(x)=o(|x|^{\frac{1-N}{2}})$, as $|x|\to\infty$.
Here $H^{(1)}_{\frac{N-2}{2}}$ denotes the Hankel function of the first kind of order $\frac{N-2}{2}$. Moreover, we wish to derive corresponding nonvanishing results in the spirit of \cite[Theorem 3.1]{Evequoz2015}. Finally, we wish to deduce existence results for real-valued $G$-invariant solutions of nonlinear Helmholtz equations of the form 
\begin{equation}\label{eqn:intro_helmholtz}
- \Delta u - u = Q(x)|u|^{p-2}u, \quad u \in W^{2,p}(\RN).
\end{equation}
With regard to our first aim, we focus our attention to the subgroups 
\begin{equation}
  \label{eq:def-G-k}
G_k:= O(N-k) \times O(k)\; \subset \;O(N) \qquad \text{for $k = 1,\dots,N-1$.}
\end{equation}
Moreover, we consider weight functions of the form $Q_\alpha = \mathds{1}_{L_\alpha}$ for the set
\begin{equation}
\label{eqn:L}
L_\alpha:= \{x = (x^{(N-k)},x^{(k)}) \in \R^{N-k} \times \R^{k} : |x^{(N-k)}| \leq a |x^{(k)}|^{-\alpha} \}, 
\end{equation}
where $a>0$ is an arbitrary fixed number and $\alpha>0$. Since $|L_\alpha|= \infty$, we have $Q_\alpha \not \in L^s(\R^N)$ for any $\alpha>0$, $s <\infty$. 

\begin{theorem}\label{theo:intro1-alpha-equals-beta}
  Let $N \ge 3$, $k \in \{1,\dots,N-1\}$, let $\alpha>0$, and let $Q_\alpha=\mathds{1}_{L_\alpha}$ with $L_\alpha$ given in \eqref{eqn:L}. Moreover, suppose that
  \begin{equation}
    \label{eq:extra-assumption-1-N-1}
\alpha> \frac{1}{N-1} \quad \text{if $k=1$,}\qquad \qquad \qquad 
\alpha< N-1 \quad \text{if $k=N-1$},
  \end{equation}
and let
  \begin{equation}
    \label{eq:def-lambda-n-k-alpha}
    \lambda_{N,k,\alpha}:= \left\{
      \begin{aligned}
      &\frac{2(N-1) - \frac{2}{\alpha}}{N-2},&&\qquad \text{if $k=1$;}\\  
      &\max \left\lbrace \frac{2(N-k)- \frac{2k}{\alpha}}{N-k-1},\frac{2k-2\alpha(N-k)}{k-1}\right\rbrace &&\qquad \text{if $2 \le k \le N-2$;}\\
     &\frac{2(N-1) - 2\alpha}{N-2},&& \qquad \text{if $k = N-1$.}
    \end{aligned}
  \right.
\end{equation}
Then $(G_k,q,Q_\alpha)$ is an admissible extension triple for every $q > \lambda_{N,k,\alpha}$.
      \end{theorem}
      We note that, in Theorem~\ref{theo:alpha-diff-beta-section} below, we shall in fact prove a generalization of this result for characteristic functions of sets of the form $L_{\alpha,\beta}:= \{x \in \R^N : |x^{(N-k)}| \leq a \max \{|x^{(k)}|^{-\alpha},|x^{(k)}|^{-\beta} \} \}$ with $\alpha \ge \beta>0$. Regarding Theorem~\ref{theo:intro1-alpha-equals-beta}, we note in particular that $\lambda_{N,k,\alpha}=0$ for $\alpha = \frac{k}{N-k}$, so $(G_k,q,Q_\alpha)$ is an admissible extension triple for every $q \ge 1$ in this case if also (\ref{eq:extra-assumption-1-N-1}) is satisfied. More generally, the latter property holds if $\alpha \in (\frac{k+1}{2(N-k)}, \frac{2k}{N-k+1})$, since then we have $\lambda_{N,k,\alpha}<1$. Comparing with the classical Stein-Tomas exponent, we have $\lambda_{N,k,\alpha} < \frac{2(N+1)}{N-1}$ if
     $$
     \begin{aligned}
       &k\le \frac{N-1}{2},&&\qquad \alpha \in \bigl( \frac{N+1-2k}{(N-k)(N-1)},\infty \bigr)\qquad \text{or}\\
       &\frac{N-1}{2} < k < \frac{N+1}{2},&&\qquad \alpha \in \bigl( \frac{N+1-2k}{(N-k)(N-1)},\frac{k(N-1)}{2k-(N-1)} \bigr)\qquad \text{or}\\
      &k \ge \frac{N+1}{2},&&\qquad \alpha \in \bigl(0,\frac{k(N-1)}{2k-(N-1)}\bigr).
     \end{aligned}
$$

\begin{figure}[h]
   \centering
     	\includegraphics[width=0.57\textwidth]{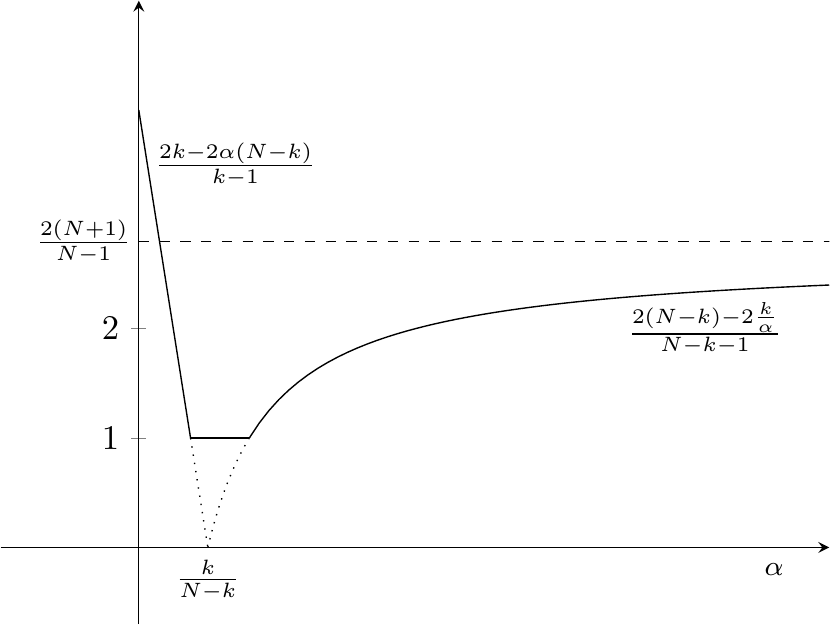}
     	\caption{Possible values of $q$  for $N=6$ and $k=2$ depending on $\alpha$.}
     \end{figure} 
     

       The main part of the proof of this Theorem consists in a detailed asymptotic study of one-dimensional integrals which arise after integrating along the orbits of $G_k$. Here, the well-known bound 
$$
|\wdc{d\sigma_k}(x)| \leq C (1+|x|)^{\frac{1-k}{2}}, \qquad x \in \R^k
$$
for the Fourier transform of the standard measure $d\sigma_k$ on $\Sp^{k-1}$ will play a key role (see e.g. \cite[\S8: Theorem 3.1]{Stein-K}). 

We also remark that, if $(G,q,Q)$ is an admissible extension triple and $Q':\R^N \to \C$ is a measurable function with $|Q'| \le |Q|$ in $\R^N$, then, by definition, $(G,q,Q')$ is also an admissible extension triple. Consequently, the statement of Theorem~\ref{theo:intro1-alpha-equals-beta} extends to functions $Q \in L^\infty(\R^N)$ with $|Q| \le c \mathds{1}_{L_\alpha}$ in $\R^N$ for some $c>0$. 

Next we state our main result on $(L^{p'},L^p)$-Helmholtz resolvent estimates for $G$-invariant functions. Here and in the following, for $r \in [1,\infty]$, we let $L^r_G(\R^N)$ denotes the closed subspace of $G$-invariant functions in $L^r(\R^N)$.

\begin{theorem}
\label{resolvent-estimates-intro}
Let $N \ge 3$, let $G \subset O(N)$ be a closed subgroup, let $Q \in L^\infty_G(\R^N)$, and let $q \in [1,\frac{2(N+1)}{N-1}]$ be such that $(G,q,Q)$ is an admissible extension triple. Then for every $p \in \bigl(\frac{2N}{N-1} \frac{2 q}{q+ 2}\,,\, \frac{2N}{N-2}\bigr]$ there exists a constant $C>0$ such that
\begin{equation}\label{eqn:resolventestimate-intro}
\left\| \cR_{\sQ}(f) \right\|_{L^{p}} \leq C \| f \|_{L^{p'}} \qquad \text{for every $f \in \cS_{G}$.}
\end{equation}
Here and in the following, $\cS_G \subset \cS$ denotes the subspace of $G$-invariant functions in the Schwartz space $\cS$.
\end{theorem}

Our proof of Theorem~\ref{resolvent-estimates-intro} is based on the strategy used in \cite{Gutierrez2004} and \cite{kenig1987}, see also \cite{Evequoz2015}. We recall that a selfdual estimate of the form (\ref{eqn:resolventestimate-intro}) has been proved in \cite{kenig1987} for the Helmholtz resolvent $\cR$ in place of $\cR_\sQ$ in the range of exponents $p \in [\frac{2(N+1)}{N-1}, \frac{2N}{N-2}\bigr]$, while corresponding non-selfdual estimates were obtained in \cite{Gutierrez2004}. Clearly, these already available $(L^{r},L^s)$-estimates for $\cR$ extend, by approximation, to the weighted resolvent $\cR_\sQ$ in the case where $Q \in L^\infty(\R^N)$. Theorem~\ref{resolvent-estimates-intro} complements
the selfdual estimate in \cite{kenig1987}, for $\cR_\sQ$ and $G$-invariant functions, in the case where $(G,q,Q)$ is an admissible extension triple for some $q< \frac{2(N+1)}{N-1}$, which is equivalent to the inequality $\frac{2N}{N-1} \frac{2 q}{q+ 2}< \frac{2(N+1)}{N-1}$. In fact, we will prove a non-selfdual generalization of Theorem~\ref{resolvent-estimates-intro} in Theorem~\ref{resolvent-estimates-non-self-dual} below.
 
Under the assumptions of Theorem~\ref{resolvent-estimates-intro}, it follows, 
by density, that the weighted resolvent $\cR_\sQ$ extends to a bounded linear operator $L^{p'}_G(\R^N) \to L^p_G(\R^N)$. In our next result we state that, under the same assumptions, a nonvanishing property in the spirit of \cite[Theorem 3.1]{Evequoz2015} holds. 
\begin{theorem}
\label{nonvanishing-intro}
Let $N \ge 3$, let $G \subset O(N)$ be a closed subgroup, let $Q \in L^\infty_G(\R^N)$, and let $q \in \bigl[1,\frac{2(N+1)}{N-1}\bigr]$ be
such that $(G,q,Q)$ is an admissible extension triple. Moreover, let $p \in \bigl(  \frac{2N}{N-1} \frac{2 q}{q+ 2}, \frac{2N}{N-2}\bigr]$. 
Then for every bounded sequence $(v_{n})_{n} \subset L^{p'}_G(\RN)$ satisfying $\Bigl| \limsup\limits_{n \to\infty} \int\limits_{\RN}v_{n} \cR_{\sQ}(v_{n})~dx\Bigr|>0$, there exist -- after passing to a subsequence -- numbers $R,\zeta>0$ and a sequence of points $(x_{n})_{n \in \N} \subset \RN$ with
\begin{equation*}
\int\limits_{B_{{R}(x_{n})}} |Q v_{n}(x)|^{p'}~dx \geq\zeta, \quad \text{ for all } n.
\end{equation*}
\end{theorem}

In the special (non-symmetric) case $G= \{\id\}$, $Q \equiv 1$, $q= \frac{2(N+1)}{N-1}$, this theorem reduces to \cite[Theorem 3.1]{Evequoz2015}. Here we note that $\frac{2N}{N-1} \frac{2 q}{q+ 2}= \frac{2(N+1)}{N-1}$ if $q= \frac{2(N+1)}{N-1}$. The general strategy of the proof of Theorem~\ref{nonvanishing-intro} is inspired by \cite[Theorem 3.1]{Evequoz2015}. However, additional difficulties, related to the fact that the multiplication with $Q \in L^\infty(\R^N)$ does not map $\cS$ into itself, lead to a somewhat more involved argument.

Theorems~\ref{resolvent-estimates-intro} and \ref{nonvanishing-intro} are useful in the study of real-valued $G$-invariant solutions of the nonlinear Helmholtz equation (\ref{eqn:intro_helmholtz})
with a real-valued weight function $Q \in L^\infty_G(\R^N)$, where $G \subset O(N)$ is a given closed subgroup. In the following, we focus on {\em dual bound state solutions}, which arise
as
solutions $u \in L^p(\R^N)$ of the integral equation $u = R\bigl(Q|u|^{p-2}u\bigr)$, where $R$ is the real part of the resolvent operator $\cR$, see Section~\ref{sec:dual-vari-fram} for details.
Our first main result in this context is the following.

\begin{theorem}
\label{dual-ground-state-intro}
Let $N \ge 3$, let $G \subset O(N)$ be a closed subgroup, and let $Q \in L^\infty_G(\R^N)$ be a real-valued nonnegative function with $Q \not \equiv 0$ and with the property that 
\begin{equation}
  \label{eq:asymptotic-Q-condition}
\|Q\|_{L^1(B_R(x))} \to 0 \qquad \text{as $|x| \to \infty$ for some $R>0$.}   
\end{equation}
Moreover, let $q \in \bigl[1,\frac{2(N+1)}{N-1}\bigr]$, and let $p \in \bigl(\max \bigl\{\frac{2N}{N-1} \frac{2 q}{q+ 2} , 2\bigr \}\,,\, \frac{2N}{N-2}\bigr)$ be such that $(G,q,Q^{\frac{1}{p}})$ is an admissible extension triple. Then (\ref{eqn:intro_helmholtz}) admits a nontrivial $G$-invariant dual bound state solution. 
\end{theorem}

We recall that, by the Stein-Tomas inequality, $(G,q,Q')$ is an admissible extension triple for $q= \frac{2(N+1)}{N-1}$ and every closed subgroup $G \subset O(N)$ and every $Q' \in L^\infty(\R^N)$. Recalling moreover that $\frac{2N}{N-1} \frac{2 q}{q+ 2}= \frac{2(N+1)}{N-1}$ in this case, we readily deduce the following corollary of Theorem~\ref{dual-ground-state-intro}. 

\begin{corr}
\label{dual-ground-state-intro-cor-2}
Let $N \ge 3$, let $G \subset O(N)$ be a closed subgroup, and let $Q \in L^\infty_G(\R^N)$ be a nonnegative function satisfying $Q \not \equiv 0$ and ~(\ref{eq:asymptotic-Q-condition}). Then (\ref{eqn:intro_helmholtz}) admits a nontrivial $G$-invariant dual bound state solution for every $p \in (\frac{2(N+1)}{N-1}, \frac{2N}{N-2}).$
\end{corr}
This corollary applies in particular in the non-symmetric case $G= \{\id\}$, but it requires the asymptotic condition (\ref{eq:asymptotic-Q-condition}). On the other hand, in the case of special symmetries corresponding to the subgroups $G_k$ defined in (\ref{eq:def-G-k}), we may drop assumption (\ref{eq:asymptotic-Q-condition}), as the following result shows.

\begin{theorem}
\label{dual-ground-state-intro-cor-3}
Let $N \ge 4$, let $k \in \{2,\dots,N-2\}$, and let $Q \in L^\infty_{G_k}(\R^N)$ be a nonnegative function with $Q \not \equiv 0$. Then (\ref{eqn:intro_helmholtz}) admits a nontrivial $G_{k}$-invariant dual bound state solution for every $p \in (  \frac{2(N+1)}{N-1}, \frac{2N}{N-2})$.
\end{theorem}

Finally, we point out that assumption (\ref{eq:asymptotic-Q-condition}) holds in particular for functions $Q \in L^\infty(\R^N)$ satisfying $|Q| \le c \mathds{1}_{L_\alpha}$ for some $c,\alpha>0$, where $L_\alpha$ is given in (\ref{eqn:L}). Using this fact, the following corollary can be deduced from Theorems~\ref{theo:intro1-alpha-equals-beta} and~\ref{dual-ground-state-intro}.

\begin{corr}
\label{dual-ground-state-intro-cor-1}
Let $N \ge 3$, let $k \in \{1,\dots,N-1\}$, and let $\alpha>0$. Moreover, let $Q \in L^\infty_{G_k}(\R^N)$ be a nonnegative function with $Q \not \equiv 0$ and satisfying $|Q| \le c \mathds{1}_{L_\alpha}$ for some $c>0$ with $L_\alpha$ given in \eqref{eqn:L}. 
Then (\ref{eqn:intro_helmholtz}) admits a nontrivial $G_{k}$-invariant dual bound state solution for all $p \in (\mu_{N,k,\alpha},\frac{2N}{N-2})  $ if one of the following holds:
\begin{itemize}
\item[(i)]  $k=1$  and 
\begin{equation}\label{eqn:explicit_k=1}
\mu_{N,1,\alpha} :=\text{\footnotesize $
\left\{
  \begin{aligned}
    &\text{\normalsize $2$}, &&\qquad \qquad \frac{1}{N-1} \text{\normalsize $\,<\, \alpha \leq$}\: \frac{N+1}{3(N-1)},\\
    &\frac{4N (\alpha(N-1)- 1)}{(N-1) (2\alpha N -3\alpha-1)},&&\qquad \qquad \text{\normalsize $\alpha \,>$}\:\frac{N+1}{3(N-1)}.
  \end{aligned}
\right.  
$}
\end{equation}
\item[(ii)] $k=N-1$ and 
  \begin{equation}
    \label{eqn:explicit_k=N-1}
\mu_{N,N-1,\alpha} :=\text{\footnotesize $
\left\{
  \begin{aligned}
    &\frac{4N(N-1-\alpha)}{(N-1)(2N-\alpha-3)}, &&\qquad \qquad \text{\normalsize $ 0\,<\alpha \leq\:$} \frac{3(N-1)}{N+1},\\
    &\text{\normalsize $2$}, &&\qquad \qquad \frac{3(N-1)}{N+1} \text{\normalsize $\,<\, \alpha \,<$}\: N-1.
  \end{aligned}
\right.  
$}
\end{equation}
\item[(iii)] $2 \le k \le N-2$ and 
  \begin{equation}
    \label{eq:def-mu-N-k-alpha}
\mu_{N,k,\alpha} :=\text{\footnotesize $
\left\{
  \begin{aligned}
    &\frac{4N(k-\alpha(N-k))}{(N-1)(2k-1-\alpha(N-k))}, &&\qquad \qquad \text{\normalsize $\alpha \leq\:$} \frac{N+2k-1}{(N+1)(N-k)},\\
    &\text{\normalsize $2$}, &&\qquad \qquad \frac{N+2k-1}{(N+1)(N-k)} \text{\normalsize $\,<\, \alpha \leq$}\: \frac{(N+1)k}{N-1 + 2(N-k)},\\
    &\frac{4N (\alpha(N-k)- k)}{(N-1) (2\alpha (N-k)-\alpha- k)},&&\qquad \qquad \text{\normalsize $\alpha \,>$}\: \frac{(N+1)k}{N-1 + 2(N-k)}.
  \end{aligned}
\right.  
$}
\end{equation}
\end{itemize}
\end{corr}

We point out that, in contrast to Corollary~\ref{dual-ground-state-intro-cor-2} and Theorem~\ref{dual-ground-state-intro-cor-3}, Corollary~\ref{dual-ground-state-intro-cor-1} allows to consider exponents $p< \frac{2(N+1)}{N-1}$.

To put our existence results for (\ref{eqn:intro_helmholtz}) into perspective, we recall some previous results. In \cite{Gutierrez2004}, the existence of small complex solutions has been proved via the use of contraction mappings in dimensions $N=3,4$, $p=3$ and $Q=\pm 1$. A variant of this technique is developed in \cite{Mandel1}, where continua of small real-valued solutions of (\ref{eqn:intro_helmholtz}) are detected for a larger class of nonlinearities. The dual variational approach to (\ref{eqn:intro_helmholtz}) was introduced in \cite{Evequoz2015}, where the existence of nontrivial dual bound state solutions was proved for $p \in \left(\frac{2(N+1)}{N-1},\frac{2N}{N-2}\right)$ and for nonnegative weight functions $Q \in L^{\infty}(\RN)\setminus\{0\}$ which are either $\Z^N$-periodic or satisfy the uniform decay assumptions $Q(x) \to 0$ as $|x| \to \infty$. Under additional restrictions on $Q$, this approach was extended to the Sobolev critical case $p = \frac{2N}{N-2}$ in \cite{Evequoz2019}. Moreover, a dual approach in Orlicz spaces was developed in \cite{Evequoz-Orlicz} to treat more general nonlinearities in (\ref{eqn:intro_helmholtz}). The defocusing case $Q \le 0$ in (\ref{eqn:intro_helmholtz}) and radial solutions are considered in \cite{Mandel}. We are not aware of any previous work where symmetries different from radial symmetry are used to extend the range of admissible exponents to values below the Stein-Tomas exponent $\frac{2(N+1)}{N-1}$ and to overcome lack of compactness issues in the context of (\ref{eqn:intro_helmholtz}). .

The paper is organized as follows. In Section~\ref{sec:fourier-estimates-g}, we derive a Fourier extension estimate for $G_k$-invariant functions, where $G_k$ is defined in (\ref{eq:def-G-k}). In particular, we prove a generalization of Theorem~\ref{theo:intro1-alpha-equals-beta}. In Section~\ref{sec:resolv-estim-g}, we provide weighted Helmholtz resolvent estimates relative to a given admissible extension triple, thereby giving the proof of Theorem~\ref{resolvent-estimates-intro}. In Section~\ref{sec:nonv-g-invar} we study related nonvanishing properties, and we give the proof of Theorem~\ref{nonvanishing-intro}.
Finally, Section~\ref{sec:dual-vari-fram} is devoted to our main existence results for dual bound state solutions of (\ref{eqn:intro_helmholtz}).

We close this introduction by fixing some notation. Throughout the paper we denote by $B_{r}(x)$ the open ball in $\RN$ with radius $r>0$ and center at $x$. 
Moreover, we set $B_{r} = B_{r}(0)$ and $\Sp^{N-1}$ for the boundary of $B_{1}=:B$. The constant $\alpha_{N}$ represents the volume of the unit ball $B_1$ in $\RN$. For any element $x \in \R^{N}$ we write $x = (x^{(N-k)},x^{(k)}) := ((x_{1},\dots, x_{N-k}) , (x_{N_k+1}, \dots , x_{N})) \in \R^{N-k} \times \R^{k}$. Moreover by $B^{(k)}$ we denote the unit ball in $\R^{k}$.
By $\mathds{1}_{L}$ we denote the characteristic function of a measurable set $L \subset \RN$. 
Furthermore, we shall indifferently denote by $\widehat{f}$ or ${\mathcal{F}}(f)$ the Fourier transform of a function in $\RN$ given by
$$ \cF{f}(\xi) = \wdh{f}(\xi) = (2\pi)^{-\frac{N}{2}}\int\limits_{\RN}e^{-ix \cdot \xi}f(x)~dx $$ and by $\wdc{F}_{\sigma}$ the inverse Fourier transform of an admissible functions $F$ defined on $\Sp^{N-1}$ via
$$ \wdc{F}_{\sigma}(x) = (2\pi)^{-\frac{N}{2}} \int\limits_{\SPN}e^{ix\cdot \omega} F(\omega)~d\sigma(\omega). $$
For $1\leq s\leq \infty$, we abbreviate the norm on $L^{s}(\RN)$ by $\left\|\cdot\right\|_{s}$. The Schwartz-class of rapidly decreasing functions on $\RN$ is denoted by $\SR$. For any $p \in (1,\infty)$ we always denote by $p':=\frac{p}{p-1}$ the Hölder conjugate of $p$.\\

\noindent{\bf Acknowledgement:}  T. Weth and T. Yesil are supported by the German Science Foundation (DFG) within the project WE-2821/5-2. The authors wish to thank the referee for his/her valuable comments and corrections.

\section{Fourier extension estimates for $G_k$-invariant functions} 
\label{sec:fourier-estimates-g}

We recall that, for a function $F \in L^2(\SPN)$, we define the (inverse) Fourier transform of $Fd\sigma$ by
$$ 
\wdc{F}_{\sigma}(x) = (2\pi)^{-\frac{N}{2}}\intl_{\SPN}e^{i\omega\cdot x}F(\omega)~d\sigma(\omega). 
$$
For $F\equiv 1$ we use the notation
$$ \wdc{d\sigma}_{N}(x) = (2\pi)^{-\frac{N}{2}}\intl_{\SPN}e^{i\omega\cdot x}~d\sigma(\omega) $$
and will often omit the dimensional index if no confusion is possible. We point out that this function satisfies the key uniform 
bound 
\begin{equation}\label{eqn:asymp_fouriermeasure}
|\wdc{d\sigma}_{N}(x) | \leq C (1+|x|)^{\frac{1-N}{2}}, \qquad x \in \RN.
\end{equation}
with a constant $C=C(N)>0$, see e.g. \cite[\S8: Theorem 3.1]{Stein-K}.

For $k \in \{1,\dots,N-1\}$, we consider the closed subgroup $G_k= O(N-k) \times O(k) \subset O(N)$. We first derive a useful expression for $\wdc{F}_{\sigma}$ in the case where $F \in \cC(\SPN)$ is $G_k$-invariant. Note that in this case $F$ only depends on one variable $r \in [0,1]$ via the function 
\begin{equation}
  \label{eq:def-h-F}
h_F: [0,1] \to \R, \qquad h_F(r):= F(r\eta,\sqrt{1-r^2}\mu)\qquad \text{for $\eta \in \Sp^{N-k-1}$, $\mu \in \Sp^{k-1}$.}
\end{equation}
 \begin{lemma}\label{lem: Fourierextensionlemma}
 	Let $k \in \{1,\dots,N-1\}$ and $F \in \cC(\SPN)$ be $G_k$-invariant. Then we have
 	\begin{equation}\label{eqn: fourierextension}
 	\wdc{F}_{\sigma}(x) = (2\pi)^{\frac{N}{2}}\frac{k\alpha_{k}}{N\alpha_{N}}\int\limits^{1}_{0}r^{N-k-1}(1-r^{2})^{\frac{k-2}{2}}h_F(r)\wdc{d\sigma}_{N-k}(rx^{(N-k)})\wdc{d\sigma}_{k}(\sqrt{1-r^2}x^{(k)})~dr
 	\end{equation}
 	with $h_F$ given in~(\ref{eq:def-h-F}). Moreover, 
 	\begin{equation*}\label{eqn:fourierextensionmodulus}
 	|\wdc{F}_{\sigma}(x)| \leq (2\pi)^{\frac{N}{2}}\frac{k\alpha_{k}}{N\alpha_{N}}\frac{\left\|F\right\|_{L^{2}(\Sp^{N-1})}}{\sqrt{|\Sp^{N-k-1}||\Sp^{k-1}|}}\left(\int\limits^{1}_{0}r^{N-k-1}(1-r^{2})^{\frac{k-2}{2}}|\wdc{d\sigma}_{N-k}(rx^{(N-k)})|^2~|\wdc{d\sigma}_{k}(\sqrt{1-r^2}x^{(k)})|^2\right)^{\frac{1}{2}}. 
      \end{equation*}
 	for all $x \in \RN.$ 
 \end{lemma}
 \begin{proof}
 	By using slice integration (see e.g. \cite[A.5]{axler2001}) we have
 	\begin{align*}
 	&\wdc{F}_{\sigma}(x) = (2\pi)^{\frac{N}{2}}\int\limits_{\SPN}e^{ix\cdot\omega}F(\omega)~d\sigma(\omega) \notag\\
 	&=
 	(2\pi)^{\frac{N}{2}}\frac{k\alpha_{k}}{N\alpha_{N}}\int\limits_{B^{(N-k)} }(1-|y|^2)^{\frac{k-2}{2}}e^{ix^{(N-k)}y}\int\limits_{\Sp^{k-1}}e^{ix^{(k)}\sqrt{1-|y|^2}\mu}F\left(y,\sqrt{1-|y|^2}\mu\right)~d\sigma_{k}(\mu)~d_{N-k}(y) \\
 	&=
 	(2\pi)^{\frac{N}{2}}\frac{k\alpha_{k}}{N\alpha_{N}}\int\limits^{1}_{0}r^{N-k-1}(1-r^{2})^{\frac{k-2}{2}}\int\limits_{\Sp^{N-k-1}}e^{ix^{(N-k)}r\eta}\int\limits_{\Sp^{k-1}}e^{ix^{(k)}\sqrt{1-r^2}\mu}F\left(r\eta,\sqrt{1-r^2}\mu\right)~d\sigma_{k}(\mu)~d\sigma_{N-k}(\eta)~dr\\
 	&= 
 	(2\pi)^{\frac{N}{2}}\frac{k\alpha_{k}}{N\alpha_{N}}\int\limits^{1}_{0}r^{N-k-1}(1-r^{2})^{\frac{k-2}{2}}h_F(r)\wdc{d\sigma}_{N-k}(rx^{(N-k)})\wdc{d\sigma}_{k}(\sqrt{1-r^2}x^{(k)})~dr
 	\end{align*}
 	for all $x \in \RN$ with $h_F$ given in (\ref{eq:def-h-F}), as claimed in (\ref{eqn: fourierextension}). In particular, we get
 	\begin{align*}
 	&|\wdc{F}_{\sigma}(x)| \leq 
 	(2\pi)^{\frac{N}{2}}\frac{k\alpha_{k}}{N\alpha_{N}}\int\limits^{1}_{0}r^{N-k-1}(1-r^{2})^{\frac{k-2}{2}}|h_F(r)|~|\wdc{d\sigma}_{N-k}(rx^{(N-k)})|~|\wdc{d\sigma}_{k}(\sqrt{1-r^2}x^{(k)})|~dr.\\
 	&\leq
 	(2\pi)^{\frac{N}{2}}\frac{k\alpha_{k}}{N\alpha_{N}}\left(\int\limits^{1}_{0}r^{N-k-1}(1-r^{2})^{\frac{k-2}{2}}|h_F(r)|^{2}dr \right)^{\frac{1}{2}} \times\\ 
          &\;\qquad \left(\int\limits^{1}_{0}r^{N-k-1}(1-r^{2})^{\frac{k-2}{2}}|\wdc{d\sigma}_{N-k}(rx^{(N-k)})|^2~|\wdc{d\sigma}_{k}(\sqrt{1-r^2}x^{(k)})|^2dr\right)^{\frac{1}{2}} \end{align*}
        \begin{align*}
 	&=
 	(2\pi)^{\frac{N}{2}}\frac{k\alpha_{k}}{N\alpha_{N}}\frac{\left\|F\right\|_{L^{2}(\Sp^{N-1})}}{\sqrt{|\Sp^{N-k-1}||\Sp^{k-1}|}}\left(\int\limits^{1}_{0}r^{N-k-1}(1-r^{2})^{\frac{k-2}{2}}|\wdc{d\sigma}_{N-k}(rx^{(N-k)})|^2~|\wdc{d\sigma}_{k}(\sqrt{1-r^2}x^{(k)})|^2dr\right)^{\frac{1}{2}}. 
 	\end{align*}
 \end{proof}

For $\alpha \ge \beta>0$ and fixed $a>0$, we now consider the subset 
\begin{equation}
  \label{eq:def-L-alpha-beta}
L_{\alpha,\beta}:= \{x = (x^{(N-k)},x^{(k)}) : |x^{(N-k)}| \leq a \max\left\lbrace|x^{(k)}|^{-\alpha},|x^{(k)}|^{-\beta} \right\rbrace \} \subset \R^N.
\end{equation}

We shall prove the following generalization of Theorem~\ref{theo:intro1-alpha-equals-beta}. 

\begin{theorem}\label{theo:alpha-diff-beta-section}
	Let $N \geq 3.$ Suppose we have $\alpha \ge \beta >0$, $k \in \{1,\dots,N-1\}$, $Q=\mathds{1}_{L_{\alpha,\beta}}$ with $L_{\alpha,\beta}$ given as in \eqref{eqn:L}, and
        \begin{equation}
          \label{eq:def-lambda-N-k-alpha-beta}
\lambda_{N,k,\alpha,\beta}:= \max \left\lbrace\frac{2k-2\beta(N-k)}{k-1},\frac{2(N-k)- \frac{2k}{\alpha}}{N-k-1}
\right\rbrace.
        \end{equation}
Suppose furthermore that $q \ge 1$ and $k \in \{1,\dots,N-1\}$ satisfy
\begin{equation}
  \label{eq:assumption-alpha-diff-beta}
\left \{
  \begin{aligned}
  &k = 1,&&\qquad \beta>\frac{1}{N-1}&&\text{ and }~ q > \frac{2(N-1) - \frac{2}{\alpha}}{N-2} \text{ or}\\  
  &2 \le k \le N-2, &&\qquad q> \lambda_{N,k,\alpha,\beta} &&\text{or}\\
  &k=N-1,&&\qquad \alpha < N-1,&&\text{and }q > \frac{2(N-1)- 2\beta}{N-2}.
  \end{aligned}
\right.
\end{equation}
Then there exists a constant $C=C(N,k,\alpha,\beta,a)$ with the property that 
	\begin{equation}\label{eqn:weightedfouriermodifiedlayerintro}
	\left\|Q \wdc{F}_{\sigma} \right\|_{q} \leq C \left\|F\right\|_{L^{2}(\SPN)} \qquad \text{for every $G_k$-invariant function $F \in \cC(\SPN)$.} 
	\end{equation}
	\end{theorem}
        \begin{proof}
          We shall prove
          (\ref{eqn:weightedfouriermodifiedlayerintro}) under the additional assumption
          \begin{equation}
            \label{eq:q-less-four-add-assumption}
            q<4,
          \end{equation}
noting that for any $\alpha \ge \beta >0$ and $k \in \{1,\dots,N-1\}$, the set of values $q$ satisfying (\ref{eq:assumption-alpha-diff-beta}) and (\ref{eq:q-less-four-add-assumption}) is non-empty. Moreover, by interpolating with the trivial estimate
$$
\left\| Q \wdc{F}_{\sigma} \right\|_{\infty} \leq (2\pi)^{-N}\sqrt{|\Sp^{N-1}|} \left\|F\right\|_{L^{2}(\Sp^{N-1})},
$$
we can remove the extra assumption (\ref{eq:q-less-four-add-assumption}) a posteriori.

In the following, the letter $C$ stands for positive and possibly different constants depending only on $N,k$, $\alpha,\beta$ and $a$. Let $F \in \cC(\SPN)$ be a $G_k$-invariant function. Without loss of generality we assume that $\|F\|_{L^{2}(\SPN)} = 1$.  Using Lemma \ref{lem: Fourierextensionlemma}, we can write
	\begin{align}
	&\left\| Q \wdc{F}_{\sigma} \right\|_{q}^q= \int\limits_{\RN} |[Q \wdc{F}_{\sigma}](x)|^{q}~dx =\int\limits_{\RN}Q(x)|\wdc{F}_{\sigma}(x)|^{q}~dx \nonumber \\
		&\leq
		C\int\limits_{\RN}Q(x) \int\limits^{1}_{0}r^{(N-k-1)\frac{q}{2}}(1-r^2)^{\frac{k-2}{2}\frac{q}{2}}|\wdc{d\sigma}_{N-k}(rx^{(N-k)})|^{q}|\wdc{d\sigma}_{k}(\sqrt{1-r^2}x^{(k)})|^{q}~dr~dx \nonumber\\
		&=
		C\int\limits^{1}_{0}r^{(N-k-1)\frac{q}{2}}(1-r^2)^{\frac{k-2}{2}\frac{q}{2}}\int\limits_{\RN}Q(x)|\wdc{d\sigma}_{N-k}(rx^{(N-k)})|^{q}|\wdc{d\sigma}_{k}(\sqrt{1-r^2}x^{(k)})|^{q}~dx~dr \nonumber\\
		&=
		C\int\limits^{1}_{0}r^{(N-k-1)\frac{q}{2} -(N-k)}(1-r^2)^{\frac{q}{4}(k-2) - \frac{k}{2}}H_k(r)dr, \label{first-est-involv-H} 
	\end{align}
where we have used that $|Q|^{q} = Q$ due to the special choice of $Q$ as a characteristic function. In the last line we have set
$$
H_k(r):= \int\limits_{\R^{N}}Q\left(\frac{x^{(N-k)}}{r},\frac{x^{(k)}}{\sqrt{1-r^2}}\right)|\wdc{d\sigma}_{N-k}(x^{(N-k)})|^{q}~|\wdc{d\sigma}_{k}(x^{(k)})|^{q}~dx.
$$
Using the definition of $Q$ and the estimate (\ref{eqn:asymp_fouriermeasure}), we deduce that
	\begin{align}
H_k(r) = \int\limits_{\R^{k}}& |\wdc{d\sigma}_{k}(x^{(k)})|^{q} \!\!\!\! \int\limits_{|x^{(N-k)}| \,\leq \,a r \max \bigl\{ \left(\frac{|x^{(k)}|}{\sqrt{1-r^2}}\right)^{-\alpha}, \left(\frac{|x^{(k)}|}{\sqrt{1-r^2}}\right)^{-\beta} \bigr\} } |\wdc{d\sigma}_{N-k}(x^{(N-k)})|^{q}~dx^{(N-k)}~dx^{(k)} \nonumber \\
		\leq C \int\limits_{\R^{k}}& (1+|x^{(k)}|)^{q\frac{1-k}{2}} \!\!\!\! \int\limits_{|x^{(N-k)}| \,\leq \,a r \max \bigl\{ \left(\frac{|x^{(k)}|}{\sqrt{1-r^2}}\right)^{-\alpha}, \left(\frac{|x^{(k)}|}{\sqrt{1-r^2}}\right)^{-\beta} \bigr\} } (1+|x^{(N-k)}|)^{q\frac{1-(N-k)}{2}}~dx^{(N-k)}~dx^{(k)} \nonumber \\
		=
           C \int\limits^{\infty}_{0} &s^{k-1}(1+s)^{-\frac{q}{2}(k-1)} \int\limits^{ar \max \bigl\{\bigl(\frac{\sqrt{1-r^2}}{s}\bigr)^{\alpha},\bigl(\frac{\sqrt{1-r^2}}{s}\bigr)^{\beta} \bigr\}}_{0} t^{N-k-1}(1+t)^{-\frac{q}{2}(N-k-1)}~dt~ds \nonumber\\
           =
             C &\int\limits^{\sqrt{1-r^2}}_{0} s^{k-1}(1+s)^{-\frac{q}{2}(k-1)} \int\limits^{ar (1-r^2)^{\frac{\alpha}{2}}s^{-\alpha} }_{0} t^{N-k-1}(1+t)^{-\frac{q}{2}(N-k-1)}~dt~ds \nonumber\\
          &+\; C \int\limits^{\infty}_{\sqrt{1-r^2}}s^{k-1}(1+s)^{-\frac{q}{2}(k-1)} \int\limits^{ar (1-r^2)^{\frac{\beta}{2}}s^{-\beta}}_{0} t^{N-k-1}(1+t)^{-\frac{q}{2}(N-k-1)}~dt~ds, \label{scnd-est-involv-H}
	\end{align}
	where the last equality follows since $\alpha \ge \beta>0$ by assumption and therefore
	$$ 
\max \bigl\{ \left(\frac{s}{\sqrt{1-r^2}}\right)^{-\alpha}, \left(\frac{s}{\sqrt{1-r^2}}\right)^{-\beta} \bigr\}
= \begin{cases} 	(1-r^2)^{\frac{\alpha}{2}}s^{-\alpha} \qquad &  0< s <\sqrt{1-r^2};\\
  (1-r^2)^{\frac{\beta}{2}}s^{-\beta}, \qquad &s \geq \sqrt{1-r^2}.
\end{cases}
$$
Combining \eqref{first-est-involv-H} and (\ref{scnd-est-involv-H}), we can estimate     
\begin{equation}
  \label{eq:Q-I-i-est}
\left\| Q \wdc{F}_{\sigma} \right\|_{q}^q \leq C \Bigl(I^{(1)}_k+I^{(2)}_k\Bigr)
\end{equation}
with
\begin{equation}
  \label{eq:def-I-k-i}
I^{(i)}_k:= \int\limits^{1}_{0} r^{(N-k-1)\frac{q}{2} -(N-k)}(1-r^2)^{\frac{q}{4}(k-2) - \frac{k}{2}}H^{(i)}_k(r)~dr \qquad\text{for $i=1,2$}
\end{equation}
and
\begin{align*}
	H^{(1)}_{k}(r)&:=\int\limits^{\infty}_{\sqrt{1-r^2}}s^{k-1}(1+s)^{-\frac{q}{2}(k-1)} \int\limits^{ar (1-r^2)^{\frac{\beta}{2}}s^{-\beta}}_{0} t^{N-k-1}(1+t)^{-\frac{q}{2}(N-k-1)}~dt~ds,\\
	H^{(2)}_{k}(r)&:=\int\limits^{\sqrt{1-r^2}}_{0}s^{k-1}(1+s)^{-\frac{q}{2}(k-1)} \int\limits^{ar (1-r^2)^{\frac{\alpha}{2}}s^{-\alpha}}_{0} t^{N-k-1}(1+t)^{-\frac{q}{2}(N-k-1)}~dt~ds.	
	\end{align*}
        We first estimate $I^1_{k}$, and we note that
\begin{align}
  H_k^{(1)}(r)& \le \int \limits^{\infty}_{\sqrt{1-r^2}}s^{(1-\frac{q}{2})(k-1)} \int\limits^{ar (1-r^2)^{\frac{\beta}{2}}s^{-\beta}}_{0} t^{N-k-1}~dt~ds \nonumber \\
  & \le C r^{N-k}(1-r^2)^{\frac{\beta}{2}(N-k)} \int \limits^{\infty}_{\sqrt{1-r^2}}s^{(1-\frac{q}{2})(k-1)-\beta(N-k)}~ds \le C r^{N-k}(1-r^2)^{(1-\frac{q}{2})\frac{k-1}{2}+\frac{1}{2}} \label{H-k-1-r-est}
\end{align}
for $0<r<1$. Here we used in the last step that
$(1-\frac{q}{2})(k-1)-\beta(N-k)<-1$ since by (\ref{eq:assumption-alpha-diff-beta}) we have
$$
\beta >\frac{1}{N-1}\quad \text{in case $k=1$}\qquad \text{and}\qquad  q>\frac{2k-2\beta(N-k)}{k-1}\quad \text{if $2 \le k \le N-1$.}
$$
Combining (\ref{eq:def-I-k-i}) and (\ref{H-k-1-r-est}), we conclude that 
$$
I^{(1)}_k \le C \int\limits^{1}_{0} r^{(N-k-1)\frac{q}{2}}(1-r^2)^{\frac{q}{4}(k-2)- \frac{k}{2} + (1-\frac{q}{2})\frac{k-1}{2}+\frac{1}{2}}dr  = 
C \int\limits^{1}_{0} r^{(N-k-1)\frac{q}{2}}(1-r^2)^{-\frac{q}{4}}dr,
$$
where the integral on the RHS is finite due to (\ref{eq:q-less-four-add-assumption}). As a consequence, $I^{(1)}_{k} \le C_{1}$ with a constant $C_{1}=C_{1}(N,k,\beta,a)>0$.\\
Next we consider $I^{(2)}_k$, and we distinguish the following cases.\\[0.1cm]
{ \bf Case 1:} $q(N-k-1) < 2(N-k)$.\\
In this case we estimate as follows:
\begin{align*}
	H^{(2)}_{k}(r)&\le \int\limits^{\sqrt{1-r^2}}_{0}s^{k-1} \int\limits^{ar (1-r^2)^{\frac{\alpha}{2}}s^{-\alpha}}_{0} t^{N-k-1-\frac{q}{2}(N-k-1)}~dt~ds\\ 	
&\leq  C r^{N-k-\frac{q}{2}(N-k-1)}(1-r^2)^{\frac{\alpha}{2}(N-k) - \frac{\alpha}{2}\frac{q}{2}(N-k-1)}\int\limits^{\sqrt{1-r^2}}_{0}s^{k-1-\alpha(N-k) + \alpha\frac{q}{2}(N-k-1)}~ds \\
&=
Cr^{N-k-\frac{q}{2}(N-k-1)}(1-r^2)^{\frac{k}{2}} \qquad \text{for $r \in (0,1)$.}
\end{align*}
Here we used in the last step that $k-1-\alpha(N-k) + \alpha\frac{q}{2}(N-k-1) >-1$ since by (\ref{eq:assumption-alpha-diff-beta}) we have
$$
q>\frac{2(N-k)- \frac{2k}{\alpha}}{N-k-1}\quad \text{in case $k<N-1$}\qquad \text{and}\qquad \alpha < N-1 \quad \text{if $k=N-1$.}
$$
We thus conclude that 
\begin{equation}
  \label{eq:I-2-k-first-est}
I^{(2)}_k = \int\limits^{1}_{0}r^{(N-k-1)\frac{q}{2} - (N-k)}(1-r^2)^{\frac{q}{4}(k-2) - \frac{k}{2}}H^{(2)}_{k}(r)~dr \leq
C\int\limits^{1}_{0}(1-r^2)^{\frac{q}{4}(k-2)}~dr < \infty.
\end{equation}
We now consider\\[0.1cm]
{\bf Case 2:} $q(N-k-1) \ge 2(N-k)$.\\
In this case we choose $\delta \in \bigl(\,\frac{q}{2}(N-k-1)-(N-k)\:,\: \min \bigl\{ \frac{q}{2}(N-k-1), \frac{q}{2}(N-k-1)+\frac{k}{\alpha}-(N-k)\bigr \}\,\bigr)$, and we estimate as follows: 
\begin{align*}
  &H^{(2)}_{k}(r)\le
\int\limits^{\sqrt{1-r^2}}_{0}s^{k-1} \int\limits^{ar (1-r^2)^{\frac{\alpha}{2}}s^{-\alpha}}_{0} t^{N-k-1}(1+t)^{\delta-\frac{q}{2}(N-k-1)}~dt~ds\\	
&\le              \int\limits^{\sqrt{1-r^2}}_{0}s^{k-1} \int\limits^{ar (1-r^2)^{\frac{\alpha}{2}}s^{-\alpha}}_{0} t^{N-k-1+\delta-\frac{q}{2}(N-k-1)}~dt~ds\\ 	
&\leq  C r^{N-k+\delta -\frac{q}{2}(N-k-1)}(1-r^2)^{\frac{\alpha}{2}(N-k+\delta) - \frac{\alpha}{2}\frac{q}{2}(N-k-1)}\int\limits^{\sqrt{1-r^2}}_{0}s^{k-1-\alpha(N-k+\delta) + \alpha\frac{q}{2}(N-k-1)}~ds \\
&=
Cr^{N-k+\delta-\frac{q}{2}(N-k-1)}(1-r^2)^{\frac{k}{2}} \qquad \text{for $r \in (0,1)$.}
\end{align*}
We thus conclude that
\begin{equation}
  \label{eq:I-2-k-second-est}
I^{(2)}_k = \int\limits^{1}_{0}r^{(N-k-1)\frac{q}{2} - (N-k)}(1-r^2)^{\frac{q}{4}(k-2) - \frac{k}{2}}H^{(2)}_{k}(r)~dr \leq
C\int\limits^{1}_{0}r^\delta (1-r^2)^{\frac{q}{4}(k-2)}~dr < \infty.
\end{equation}
Combining (\ref{eq:I-2-k-first-est}) and (\ref{eq:I-2-k-second-est}), we conclude that  
$I^{(2)}_{k} \le C_{2}$ with a constant $C_{2}=C_{2}(N,k,\alpha,a)>0$.\\
Going back to \eqref{eq:Q-I-i-est}, we deduce that
$$
\left\| Q \wdc{F}_{\sigma} \right\|_{q}^q \leq C(C_1 + C_2), 
$$
where the constant on the RHS only depends on $N,k,\alpha,\beta$ and $a$. The proof is thus finished.
\end{proof}

We note that Theorem~\ref{theo:intro1-alpha-equals-beta} is a direct consequence of Theorem~\ref{theo:alpha-diff-beta-section}, since the assumptions of Theorem~\ref{theo:intro1-alpha-equals-beta} imply those of Theorem~\ref{theo:alpha-diff-beta-section} in the case $\alpha= \beta$. 

Moreover, we have the following duality property. 

\begin{lemma}
\label{dual-restriction-est}
Suppose that a closed subgroup $G \subset O(N)$, $q \ge 1$, and $Q \in L^\infty_G(\R^N)$ are given with the property that $(G,q,Q)$ is an admissible extension pair in the sense of Definition~\ref{sec:introduction-definition}. Then there exists a constant $C>0$ with 
$$
\bigl\|\wdh{Qf}\big|_{\SPN}\bigr\|_{L^2(\SPN)} \le C \|f\|_{q'} \qquad \text{for every $f \in \cS_G$.}
$$
In particular, this holds if $G=G_k$ and $N,k,\alpha,\beta,q$ and $Q$ satisfy the assumptions of Theorem~\ref{theo:alpha-diff-beta-section}.
\end{lemma}

\begin{proof}
Let $f \in \cS_G$ and $F:= \wdh{Qf}\big|_{\SPN} \in L^2(\SPN)$. Then we have 
\begin{align*}
&\|F\|_{L^2(\SPN)}^2 = \int\limits_{\SPN}\wdh{Qf}\overline{F}~d\sigma
	=
	(2\pi)^{-\frac{N}{2}}\int\limits_{\SPN} \int\limits_{\RN}e^{-ix\theta}f(x)Q(x)~dx~\overline{F(\theta)}~d\sigma(\theta) \\
	&=(2\pi)^{-\frac{N}{2}} \int\limits_{\RN}f(x) \overline{\overline{Q(x)}\int\limits_{\SPN}e^{ix\theta}F(\theta)~d\sigma(\theta)}~dx \le \|f\|_{q'} \|Q \check{F_\sigma}\|_{q} \le C \|f\|_{q'} \|F\|_{L^2(\SPN)}
\end{align*}
and therefore $\|F\|_{L^2(\SPN)}\le C \|f\|_{q'}$, as claimed.
\end{proof}

\section{Resolvent estimates for $G$-invariant functions}
\label{sec:resolv-estim-g}

For $N \geq 3$, the radial outgoing fundamental solution of the Helmholtz equation $-\Delta u-u=\delta_0$ in $\RN$ is given by
\begin{equation}\label{fundsol-section}
\Phi(x):= \frac{\ri}{4}(2\pi|x|)^{\frac{2-N}{2}}H^{(1)}_{\frac{N-2}{2}}(|x|), \quad\text{for }x\in\RN\backslash\{0\},
\end{equation}
where $H^{(1)}_{\frac{N-2}{2}}$ denotes the Hankel function of the first kind of order $\frac{N-2}{2}$.
For a function $f \in \cS(\RN)$ the convolution $u := \Phi\ast f \in \mathcal{C}^{\infty}(\RN)$ is a solution 
of the inhomogeneous Helmholtz equation $-\Delta u - u = f$ which satisfies the Sommerfeld 
outgoing radiation condition $\partial_ru(x)-\ri u(x)=o(|x|^{\frac{1-N}{2}})$, as $|x|\to\infty$.
Moreover, it is known (see \cite{Gelfand1964}) that, in the sense of tempered distributions, the Fourier transform of $\Phi$
is given by
\begin{equation}\label{eqn:Fourier_Phi}
\wdh{\Phi}(\xi)=(2\pi)^{-\frac{N}2}\frac{1}{|\xi|^2-(1+\ri 0)}:=(2\pi)^{-\frac{N}2}\lim_{\eps\to 0^+}\frac{1}{|\xi|^2-(1+\ri \eps)}.
\end{equation}
As a consequence of a classical estimate of Kenig, Ruiz and Sogge (see \cite[Theorem 2.3]{kenig1987}), the mapping $f\mapsto  \Phi \ast f$ for $f \in \SR(\RN)$ extends as a continuous linear operator
$$ \cR : L^{p'}(\RN) \to L^{p}(\RN) $$
for $\frac{2(N+1)}{N-1} \leq p \leq \frac{2N}{N-2}$. Moreover, non-selfdual $(L^{r},L^{p'})$-estimates for $\cR$ were established by Guti\'errez in \cite[Theorem 6]{Gutierrez2004}. The aim of this section is to establish a similar estimate for the operator $\cR_{\sQ}$ defined by $\cR_{\sQ}(f):= f \mapsto Q [\Phi \ast (Qf)]$, where $G$ is a closed subgroup of $O(N)$ and $Q \in L^\infty_G(\R^N)$ is a weight function. The main result of this section is the following generalization of Theorem~\ref{resolvent-estimates-intro}. 
\begin{theorem}
\label{resolvent-estimates-non-self-dual}
Let $N \ge 3$, let $G \subset O(N)$ be a closed subgroup, let $Q \in L^\infty_G(\R^N)$, and let $q \in \bigl[1,\frac{2(N+1)}{N-1}\bigr]$ be such that $(G,q,Q)$ is an admissible extension triple. Moreover, let $p,r \in (1,\infty)$ satisfy
\begin{equation}
  \label{eq:p-r-condition-non-self-dual-general}
  \frac{N-2}{N} \: \le \: \frac{1}{p}+\frac{1}{r} \: <  \: \frac{q+2}{2q}\,\frac{N-1}{N}.
\end{equation}  
Suppose moreover that 
\begin{equation}
  \label{eq:p-r-condition-non-self-dual-q-greater-2}
  \max\left\lbrace\frac{1}{r}, \frac{1}{p}\right\rbrace < \frac{N-1}{2N} \qquad \text{if $q \ge 2$,}
\end{equation}
and that 
\begin{equation}
  \label{eq:p-r-condition-non-self-dual-q-smaller-2}
  \frac{2q}{(N-1)(2-q)}\frac{1}{p}- \frac{(N-1)q-2(N-3)}{2N(2-q)}  < \frac{1}{r}< \frac{(N-1)(2-q)}{2q} \frac{1}{p}+ \frac{(N-1)[(N-1)q-2(N-3)]}{4qN}
  \end{equation}
if $q<2$. Then there exists  $C>0$ such that 
\begin{equation}
\label{resolvent-estimates-non-self-dual-conclusion}
        \left \|\cR_\sQ f \right\|_{r} \leq C \left\|f\right\|_{p'} \qquad \text{for all functions $f \in \cS_{G}(\RN)$.}
\end{equation}
\end{theorem}
\begin{remark}
  \label{non-self-dual}
  (i) In the special case $r = p$, the assumptions of Theorem~\ref{resolvent-estimates-non-self-dual} reduce to \mbox{$p \in \bigl(\frac{2N}{N-1}\,\frac{2q}{q+2}\,,\, \frac{2N}{N-2}\bigr]$.} Hence Theorem~\ref{resolvent-estimates-intro} follows directly from Theorem~\ref{resolvent-estimates-non-self-dual}.\\
  (ii) The assumption $q \in \bigl[1,\frac{2(N+1)}{N-1}\bigr]$ implies that condition (\ref{eq:p-r-condition-non-self-dual-general}) covers the (nonempty) condition $\frac{N-2}{N} \le  \frac{1}{p}+\frac{1}{r} <  \frac{N-1}{N+1}$ considered in \cite[Theorem 6]{Gutierrez2004}.\\
  (iii) Geometrically, the conditions (\ref{eq:p-r-condition-non-self-dual-general}), (\ref{eq:p-r-condition-non-self-dual-q-greater-2}) and (\ref{eq:p-r-condition-non-self-dual-q-smaller-2}) can be formulated for the point $(\frac{1}{p
  },\frac{1}{r})$ to be contained in the trapezoid in $(0,1) \times (0,1)$ spanned by the points
  $$
  (\frac{N-3}{2N},\frac{N-1}{2N}), \quad (\frac{N-1}{qN},\frac{N-1}{2N}),\quad (\frac{N-1}{2N},\frac{N-3}{2N})\quad \text{and}\quad (\frac{N-1}{2N},\frac{N-1}{qN}),
  $$
  with part of the boundary being excluded due to the fact that some of these inequalities are strict. In the case $q= 2$, this trapezoid degenerates to a triangle.
\end{remark}

\begin{figure}[H]
  \centering
 \captionsetup[subfloat]{labelformat=empty}
  \subfloat[][]{\includegraphics[width=0.35\linewidth]{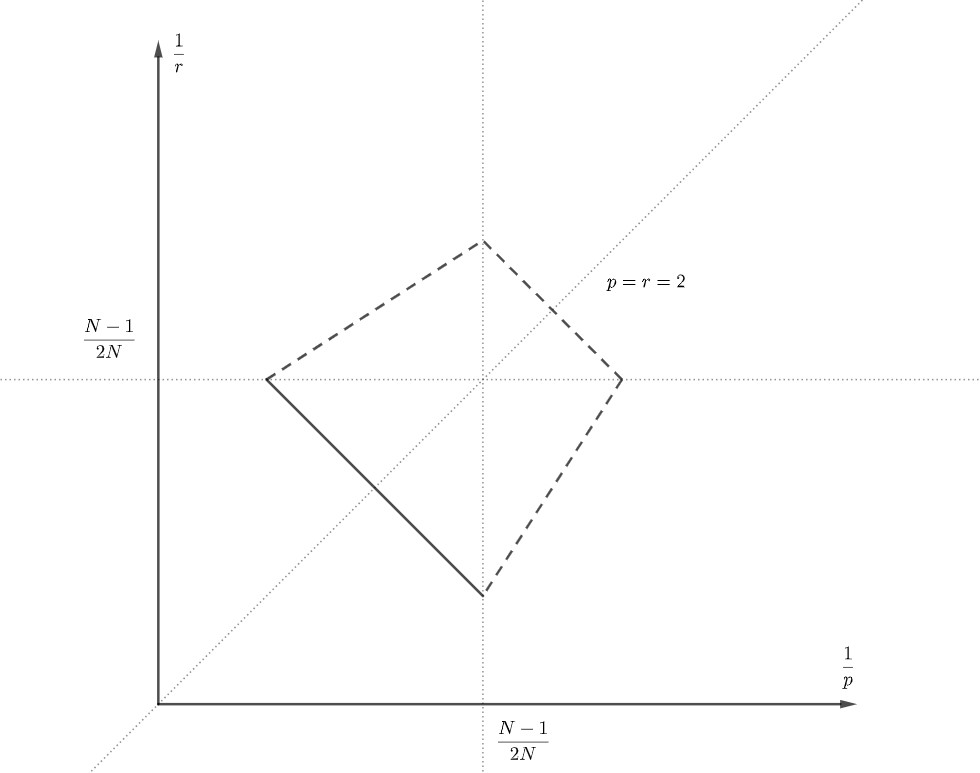}}
  \qquad
  \subfloat[][]{\includegraphics[width=0.35\linewidth]{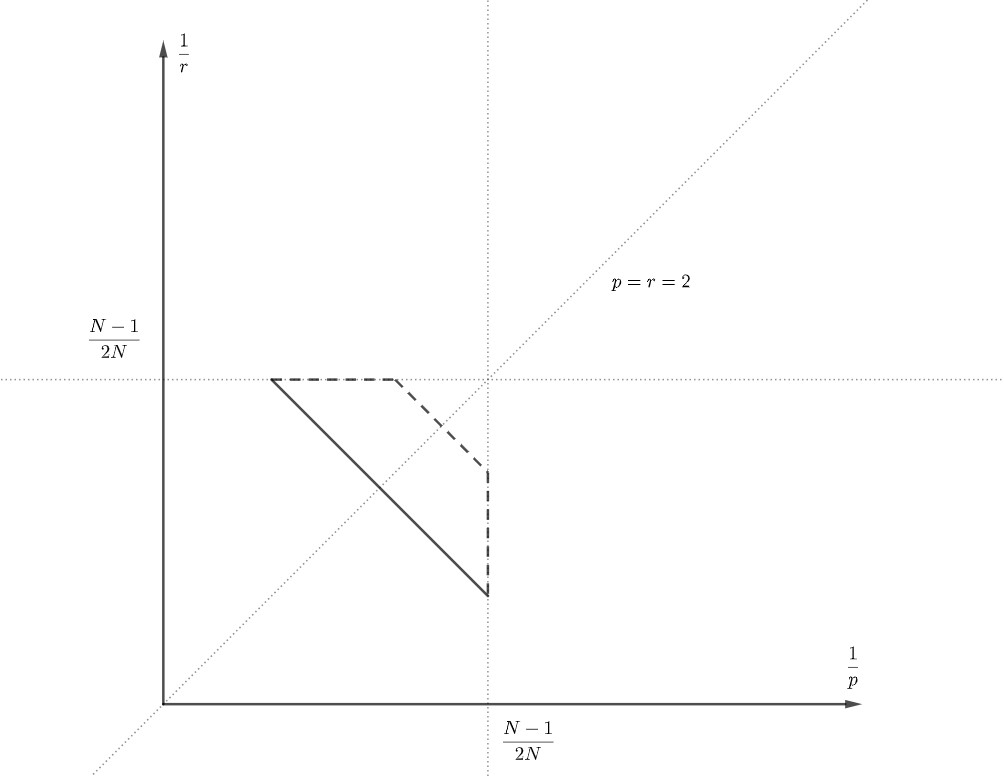}}
  \caption{Reflected Riesz diagrams for $1\leq q < 2$ (left) and $q>2$ (right)}
\end{figure}

In order to prove Theorem~\ref{resolvent-estimates-non-self-dual}, we adapt the strategy of \cite{Gutierrez2004} and \cite{kenig1987}, see also \cite{Evequoz2015}. Throughout the remainder of this section, we fix a closed subgroup $G \subset O(N)$ and $Q \in L^\infty_G(\R^N)$. We first note the following lemma which is a basic consequence of complex interpolation.

\begin{lemma}\label{lem:intpol}
 Let $1 \leq  q< \infty$ and let $\rho \in \cS$ be a radial function. Suppose furthermore that 
\begin{align}
\|Q [\rho *  (Qu)] \|_{2} &\leq C D \left\|u \right\|_{q'}&& \text{and} \label{interpol-first}\\
\|Q [\rho *  (Qu)] \|_{\infty} &\leq C D^{1-N} \left\|u\right\|_{1}  \label{interpol-second}
&& \text{for all $u \in \cS_{G}(\RN)$}
\end{align}
with constants $C,D>0$. Suppose furthermore that $p,r \in (1,\infty)$ satisfy
\begin{equation}
  \label{eq:complex-interpolation-non-selfdual-assumption}
\frac{d_{q,1}}{p}  \le \frac{1}{r} \le \min \left\lbrace\frac{1}{2}+ \frac{1}{q}-\frac{1}{p}, \frac{d_{q,2}}{p}  \right\rbrace,
\end{equation}
where
\begin{equation}
  \label{eq:def-d-q-i}
d_{q,1} = \min\left\lbrace\frac{q}{2},\frac{2}{q}\right\rbrace\qquad \text{and} \qquad d_{q,2} = \max\left\lbrace\frac{q}{2},\frac{2}{q}\right\rbrace.
\end{equation}
Then we have 
\begin{equation}
  \label{eq:lem-intpol-claim}
\left\|Q \bigl[\rho *  (Qu)\bigr]  \right\|_{r} \leq C D^{A_{r,p,q}} \left\|u\right\|_{p'} \qquad \text{for all $u \in \cS_{G}(\RN)$,} 
\end{equation}
with
\begin{equation}
  \label{eq:def-A-q-r-p}
A_{r,p,q}:= \frac{2 q N}{q+ 2}\Bigl(\frac{1}{r}+\frac{1}{p}\Bigr) -(N-1).
\end{equation}
\end{lemma}

\begin{proof}
Since $\rho \in \cS$ is radial, the convolution with $\rho$ maps $G$-invariant functions to $G$-invariant functions. Moreover, by assumption we have 
$$
\int_{\R^N} v Q \bigl[\rho *  (Qu)\bigr] dx= \int_{\R^N} u Q\bigl[\rho *  (Qv)\bigr]dx
\le \|u\|_{2} \bigl\|Q \bigl[\rho *  (Qv)\bigr]\bigr\|_{2} \le C D \|u\|_{2} \left\|v \right\|_{q'}
$$
for all $u,v \in \cS_{G}$. By duality, we therefore have
\begin{equation}
  \label{interpol-first-dual}
  \left\|Q \bigl[\rho *  (Q u)\bigr] \right\|_{q} \leq C D \left\|u \right\|_{2}\qquad \text{for all $u \in \cS_{G}$.}
\end{equation}
 Note that, by (\ref{eq:complex-interpolation-non-selfdual-assumption}), the point $(\frac{1}{p},\frac{1}{r})$ is contained in the closed symmetric triangle in $\R^2$ spanned by the points $(\frac{1}{q},\frac{1}{2})$, $(\frac{1}{2},\frac{1}{q})$ and $(0,0)$. Hence we can write
\begin{equation}
  \label{eq:convex-combination}
\frac{1}{p} = \frac{\lambda}{q} + \frac{\mu}{2}, \quad \frac{1}{r} = \frac{\lambda}{2} + \frac{\mu}{q} \quad \qquad \text{with} \quad \qquad \lambda,\mu \ge 0,\quad  \lambda+ \mu \le 1.
\end{equation}
With $\gamma:= 1- (\lambda + \mu)$, we thus have
\begin{equation*}
\frac{1}{p'}= \frac{\lambda}{q'}+ \frac{\mu}{2}+ \frac{\gamma}{1},\qquad \frac{1}{r}=\frac{\lambda}{2}+ \frac{\mu}{q} +\frac{\gamma}{\infty},
\end{equation*}
so complex interpolation of the inequalities (\ref{interpol-first}), (\ref{interpol-second}) and (\ref{interpol-first-dual}) gives 
\begin{equation}
\label{non-selfdual-almost-final}  
\left\| Q \bigl[\rho * (Q u)\bigr] \right\|_{r} \leq C D^{\lambda + \mu - \gamma(N-1)} \left\| u \right\|_{p'}= C D^{1 - \gamma N} \left\| u \right\|_{p'} \qquad  \text{for all $u \in \SR_{G}$.}
\end{equation}
Solving (\ref{eq:convex-combination}) yields 
$\lambda + \mu = \frac{2 q }{q+ 2}\Bigl(\frac{1}{r}+\frac{1}{p}\Bigr)$ and therefore $\gamma = 1-  \frac{2 q }{q+ 2}\Bigl(\frac{1}{r}+\frac{1}{p}\Bigr).$ Hence (\ref{eq:lem-intpol-claim}) follows from (\ref{non-selfdual-almost-final}).
\end{proof}

Next we decompose the fundamental solution $\Phi$ as in \cite{Gutierrez2004} and \cite{Evequoz2015}. For this we fix $\psi \in \SR(\RN)$ such that $\wdh{\psi} \in \ceinftyc$ is radial, $0\leq \wdh{\psi} \leq 1$ and $\wdh{\psi} = 1$ for $\left||\xi|-1\right| \leq \frac16$, $\wdh{\psi}(\xi) = 0$ for $\left||\xi|- 1 \right| \geq \frac14.$ We then write 
$$ 
\Phi = \Phi_{1} + \Phi_{2}\qquad \text{with}\quad  \Phi _{1} = \psi \ast \Phi, \quad \Phi_{2} = \Phi -\Phi_{1}. 
$$
Accordingly, we write 
$$
\cR_\sQ = \cR^1_\sQ + \cR^2_\sQ\qquad \text{with}\quad \cR^i_\sQ(f):= Q [\Phi_i \ast (Qf)], \quad i = 1,2.
$$
As shown in \cite[Section 3]{Evequoz2015}, we have
\begin{equation}
 \label{eq:Phi-1-estimate}
|\Phi_1(x)| \le C (1+|x|)^{\frac{1-N}{2}} \qquad \text{for $x \in \R^N$}
\end{equation}
and 
\begin{equation}
 \label{eq:Phi-2-estimate}
|\Phi_2(x)| \le C \min \{|x|^{2-N}, |x|^{-N}\} \qquad 
\text{for $x \in \R^N \setminus \{0\}.$}  
\end{equation}
In particular, by the Hardy Littlewood Sobolev inequality, we have the following. 

\begin{lemma}
\label{R-2-estimate}
For every pair of numbers $p,r \in (1,\infty)$ satisfying
$$
\frac{1}{p}+\frac{1}{r}\ge \frac{N-2}{N},
$$
the convolution operator $f \mapsto \Phi_2 * f$ defines a bounded linear map from $L^{p'}(\R^N) \to L^r(\R^N)$.

Consequently, the operator $\cR^{2}_{\sQ}$ also defines a bounded linear map from $L^{p'}(\R^N) \to L^r(\R^N)$ in this case.
\end{lemma}

Next we turn to the operator $\cR_\sQ^1$.

\begin{proposition}\label{prop:resolventestu3}
  Let $q \in \bigl[1,\infty)$ be given such that 
$(G,q,Q)$ is an admissible extension triple. Moreover, let $p,r \in (1,\infty)$ satisfy 
\begin{equation}
\label{prop:resolventestu3-assumption}  
\frac{d_{q,1}}{p}  \le \frac{1}{r} \le  \frac{d_{q,2}}{p}  \qquad \text{and}\qquad \frac{1}{p}+\frac{1}{r} <  \frac{q+2}{2q}\,\frac{N-1}{N}
\end{equation}
with $d_{q,1}, d_{q,2}$ defined in (\ref{eq:def-d-q-i}). Then there exists  $C>0$ such that 
	$$
\left \|\cR_\sQ^1 f \right\|_{r} \leq C \left\|f\right\|_{p'} \qquad \text{for all functions $f \in \cS_{G}(\RN)$.}
$$
\end{proposition}

\begin{proof}
Let $\varphi \in \SR(\RN)$ be such that $\wdh{\varphi} \in \ceinftyc(\RN)$ is radial, $0\leq \wdh{\varphi} \leq 1$ with $\wdh{\varphi} \equiv 1$ for $\left||\xi| - 1\right| \leq \frac14$ and $\wdh{\varphi} \equiv 0$ for $\left||\xi| - 1\right| \geq \frac12$. By construction of $\Phi_1$, we then have $\wdh{\Phi_1}=\wdh{\Phi_1}\wdh{\varphi}$, which means that $\Phi_1 = (2\pi)^{-\frac{N}{2}}\Phi_1 * \varphi$ 
and therefore 
$$
\cR_\sQ^1 f = Q [ \Phi_{1} \ast (Qf)]  = (2\pi)^{-\frac{N}{2}} Q [\Phi_1 \ast \varphi \ast Qf]
\qquad \text{for $f \in \cS$.} 
$$
Choose $\eta \in \ceinftyc(\RN)$ radial with $0 \le \eta \le 1$, $\eta(x) = 1$ for $|x| \le 1$, $\eta(x) = 0$ for $|x| \geq 2$. Moreover we define 
$\psi_j \in \ceinftyc(\RN)$ by $\psi_0= \eta$ and $\psi_{j}(x) = \eta(2^{-j}x) - \eta(2^{-(j-1)}x)$ for $j \in \N$, $x \in \R^N$. 
Then we have the dyadic composition 
$$
\Phi_1 = \sum_{j=0}^\infty \Phi_1^j \qquad \text{with}\quad \Phi_1^j=  \psi_j \Phi_1
$$
Using (\ref{eq:Phi-1-estimate}), we find that 
$$ 
\left\|\Phi_1^j \right\|_{\infty} \leq C 2^{-\frac{j(N-1)}{2}}, \quad \text{ for all } j,
 $$
where the constant $C>0$ is independent of $j$. Using that $\wdh{\Phi_1^j}$ is radial, we get, with Plancherel's theorem and Lemma~\ref{dual-restriction-est},
	\begin{align*}
	\left\| \bigl(\Phi_1^{j} \ast \varphi\bigr) \ast (Qf) \right\|^{2}_{2} &= C \int\limits_{\R^N} |\wdh{\Phi_1^j}(\xi)\, \wdh{\varphi}\, \wdh{Qf}(\xi)|^2~d\xi \le C \int\limits^{\frac74}_{\frac14}r^{N-1}|\wdh{\Phi_1^j}(r)|^2 \int\limits_{\SPN}|\wdh{Qf}(r\omega)|^2~d\sigma(\omega)~dr\\
	&\leq C \|\Phi_1^j\|^2_2 \left\|Qf\right\|^{2}_{q'} \leq C 2^{j} \left\|Qf\right\|^{2}_{q'} \label{eqn:stuappl} \qquad \text{for all $f \in \cS_{G}$,}
	\end{align*}
where the constant does not depend on $j$. Consequently, we thus have 
$$
\left\|Q \bigl(\Phi_1^{j} \ast \varphi\bigr) \ast (Qf) \right\|_{2} 
 \le C 2^{\frac{j}{2}} \left\|Qf\right\|_{q'} \qquad \text{for all $f \in \cS_{G}$.}
$$
Moreover, we have 
$$
\|\Phi_1^{j} \ast \varphi\|_\infty \le \|\Phi_1^j\|_\infty \|\varphi\|_1 \le C 2^{-\frac{j(N-1)}{2}}, \quad \text{ for all } j,
$$
which implies that 
$$
\left\|Q \bigl(\Phi_1^{j} \ast \varphi\bigr) \ast (Qf) \right\|_\infty \le C 
 2^{-\frac{j(N-1)}{2}} \|Qf\|_1 \le C 2^{-\frac{j(N-1)}{2}} \|f\|_1 \qquad \text{for all $f \in \cS_{G}$.}
$$
Since the assumption (\ref{prop:resolventestu3-assumption}) implies (\ref{eq:complex-interpolation-non-selfdual-assumption}), we may apply Lemma~\ref{lem:intpol} to the radial kernel $\Phi_1^{j} \ast \varphi \in \cS(\R^N)$ and deduce that 
\begin{equation}
  \label{eq:Prop-3-5-proof-intermediate}
\left\|Q \bigl(\Phi_1^{j} \ast \varphi\bigr) \ast (Qf) \right\|_r \le  
C 2^{\frac{j}{2}A_{r,p,q}} \left\|f\right\|_{p'} \qquad \text{for all $f \in \cS_{G}$}
\end{equation}
with $A_{r,p,q}$ given in (\ref{eq:def-A-q-r-p}). By assumption (\ref{prop:resolventestu3-assumption}), we have $A_{r,p,q}<0$. Since, as remarked above, $\Phi_1 = (2\pi)^{-\frac{N}{2}} \Phi_1 \ast \varphi$,
we deduce that 
\begin{align*}
\left \|\cR_\sQ^1 f \right\|_{r}&= \left \|Q \Phi_{1} \ast (Qf) \right\|_{r} = 
(2\pi)^{-\frac{N}{2}} \left \|Q (\Phi_{1} * \varphi) \ast (Qf) \right\|_{r} \\
&\le
(2\pi)^{-\frac{N}{2}} \sum_{j=0}^\infty
\left \|Q (\Phi_{1}^j * \varphi) \ast (Qf) \right\|_{r}
 \le C_0 \|f\|_{p'}  \qquad \text{for all $f \in \cS_{G}$}
\end{align*}
with $C_0 =  C
(2\pi)^{-\frac{N}{2}} \sum \limits_{j=0}^\infty 2^{\frac{j}{2} A_{r,p,q}} <\infty$. The proof is finished.
\end{proof} 

We may now complete the

\begin{proof}[Proof of Theorem~\ref{resolvent-estimates-non-self-dual}]
Let $M$ denote the set of points $(\frac{1}{p},\frac{1}{r}) \in (0,1) \times (0,1)$ such that (\ref{resolvent-estimates-non-self-dual-conclusion}) holds with some constant $C>0$. By combining Lemma~\ref{R-2-estimate} and Proposition~\ref{prop:resolventestu3}, we see that $(\frac{1}{p},\frac{1}{r}) \in M$ if 
\begin{equation}
\label{prop:resolventestu3-assumption-1}  
\frac{d_{q,1}}{p}  \le \frac{1}{r} \le  \frac{d_{q,2}}{p}  \qquad \text{and}\qquad \frac{N-2}{N} \le  \frac{1}{p}+\frac{1}{r} <  \frac{q+2}{2q}\,\frac{N-1}{N}.
\end{equation}
Hence the closure of $M$ contains the points $(\frac{N-1}{qN},\frac{N-1}{2N})$, $(\frac{N-1}{2N},\frac{N-1}{qN})$ and therefore also the line segment between these points. Moreover, by \cite[Lemma 2.2(b)]{kenig1987}, $M$ also contains the open line segment between the points $(\frac{N-3}{2N},\frac{N-1}{2N})$ and $(\frac{N-1}{2N},\frac{N-3}{2N})$.
Hence, if $q \ge 2$, complex interpolation yields that $M$ contains all points $(\frac{1}{p},\frac{1}{r}) \in (0,1) \times (0,1)$ with the property that $\frac{1}{r}, \frac{1}{p} < \frac{N-1}{2N}$ and that 
\begin{equation}
\label{parallel-cond-proof}  
\frac{N-2}{N} \: \le \: \frac{1}{p}+\frac{1}{r} \: <  \: \frac{q+2}{2q}\,\frac{N-1}{N},
\end{equation}
i.e., all points $(\frac{1}{p},\frac{1}{r}) \in (0,1) \times (0,1)$ satisfying (\ref{eq:p-r-condition-non-self-dual-general}) and \eqref{eq:p-r-condition-non-self-dual-q-greater-2}. Hence the theorem is proved in the case $q \ge 2$.\\
If $q<2$, complex interpolation yields that $M$ contains all points $(\frac{1}{p},\frac{1}{r}) \in (0,1) \times (0,1)$ satisfying (\ref{parallel-cond-proof}) and with the property that $(\frac{1}{p},\frac{1}{q})$ lies above the line through the points $(\frac{N-1}{2N},\frac{N-3}{2N})$, $(\frac{N-1}{qN},\frac{N-1}{2N})$ and below the line through the points $(\frac{N-3}{2N},\frac{N-1}{2N})$, $(\frac{N-1}{2N},\frac{N-1}{qN})$. This is precisely the set of points $(\frac{1}{p},\frac{1}{q}) \in (0,1) \times (0,1)$ satisfying (\ref{eq:p-r-condition-non-self-dual-general}) and (\ref{eq:p-r-condition-non-self-dual-q-smaller-2}). The proof is thus also finished in this case. 
\end{proof}

\section{Nonvanishing for $G$-invariant functions}
\label{sec:nonv-g-invar}

Our next aim is to deduce a nonvanishing theorem for the operator $\cR_\sQ$ and $G$-invariant functions, where again 
$G \subset O(N)$ is a closed subgroup and $Q \in L^\infty_G(\R^N)$ is a given weight function. We restate Theorem~\ref{nonvanishing-intro} for the reader's convenience. 

\begin{theorem}
\label{nonvanishing-section}
Let $N \ge 3$, let $G \subset O(N)$ be a closed subgroup, let $Q \in L^\infty_G(\R^N)$, and let $q \in \bigl[1,\frac{2(N+1)}{N-1}\bigr]$ be
such that $(G,q,Q)$ is an admissible extension triple. Moreover, let $p \in \bigl(  \frac{2N}{N-1} \frac{2 q}{q+ 2}, \frac{2N}{N-2}\bigr]$. 
Then for every bounded sequence $(v_{n})_{n} \subset L^{p'}_G(\RN)$ satisfying $\Bigl| \limsup\limits_{n \to\infty} \int\limits_{\RN}v_{n} \cR_{\sQ}(v_{n})~dx\Bigr|>0$, there exist -- after passing to a subsequence -- numbers $R,\zeta>0$ and a sequence of points $(x_{n})_{n \in \N} \subset \RN$ with
\begin{equation}\label{eqn:nonvanishingsymmetric}
\int\limits_{B_{{R}}(x_{n})} |Q v_{n}(x)|^{p'}~dx \geq\zeta, \quad \text{ for all } n.
\end{equation}
\end{theorem}

The remainder of this section is devoted to the proof of this theorem. For this we fix $N \ge 3$ and $q \in \bigl[1,\frac{2(N+1)}{N-1}\bigr]$ such that $(G,q,Q)$ is an admissible extension triple. Moreover, we 
keep using the notation of the previous section, so we write $\Phi = \Phi_1 + \Phi_2$ and $\cR_\sQ = \cR_\sQ^1 +\cR_\sQ^2$. 
We need to analyse the operators $\cR_\sQ^1$ and $\cR_\sQ^2$ separately.  We start by proving the following variant of Proposition~\ref{prop:resolventestu3} for the operator $\cR_\sQ^1$.

\begin{proposition}\label{prop:lambdaestimate}
Let $\varphi \in \SR(\RN)$ be such that $\wdh{\varphi} \in \ceinftyc(\RN)$ is radial, $0\leq \wdh{\varphi} \leq 1$ with $\wdh{\varphi} \equiv 1$ for $\left||\xi| - 1\right| \leq \frac14$ and $\wdh{\varphi} \equiv 0$ for $\left||\xi| - 1\right| \geq \frac12$.\\
Moreover, let $\eta \in \ceinftyc(\RN)$ be radial with $0 \le \eta \le 1$, $\eta(x) = 1$ for $|x| \le 1$, $\eta(x) = 0$ for $|x| \geq 2$, and let $\eta_k(x) = \eta(2^{-k}x)$ for $k \in \N$.\\
Finally, let $p \in \bigl( \frac{2N}{N-1}\frac{2q}{q + 2},\infty)$  and $A_{p,q} := \frac{N}{p}\frac{4q}{q +2} +1-N<0$. 
	Then there exists  $C>0$ such that for $k\geq 1$ we have
	$$ 
\left\|Q \Bigl([(1-\eta_k) \Phi_1]\ast  \varphi \ast (Qf)\Bigr) \right\|_{p} \leq C \frac{2^{\frac{k+1}{2} A_{p,q}}}{1-2^{\frac{A_{p,q}}{2}}}\left\|f\right\|_{p'}
$$   
	for all functions $f \in \cS_{G}$.
\end{proposition}
\begin{proof}
Let $k \in \N$. Using the given function $\eta$, we let $\Phi_1^j$, $j \in \N \cup \{0\}$ be defined as in the proof of Proposition~\ref{prop:resolventestu3}. The proof of this proposition yields, in particular, inequality (\ref{eq:Prop-3-5-proof-intermediate}) with $r=p$, which writes as  
$$
\left\|Q \bigl[\bigl(\Phi_1^{j} \ast \varphi\bigr) \ast (Qf)\bigr] \right\|_p \le C 2^{\frac{j}{2}A_{p,q}} \left\|f\right\|_{p'} \qquad \text{for $j \in \N$ and all functions $f \in \cS_{G}$}
$$
with $A_{p,q} := A{p,p,q}= \frac{N}{p}\frac{4q}{q +2} +1-N$, cf. (\ref{eq:def-A-q-r-p}). Moreover, by construction, we have the dyadic decomposition 
$$
(1-\eta_k) \Phi_1 = \sum\limits^{\infty}_{j = k+1}\Phi_1^j
$$ 
and therefore  
\begin{align*}
\left\|Q \Bigl([(1-\eta_k) \Phi_1]\ast  \varphi \ast (Qf)\Bigr) \right\|_{p} \le \sum\limits^{\infty}_{j = k+1} \left\|Q \bigl[\bigl(\Phi_1^{j} \ast \varphi\bigr) \ast (Qf)\bigr] \right\|_p &\le C \left\|f\right\|_{p'} \sum\limits^{\infty}_{j = k+1} 2^{\frac{j}{2}A_{p,q}}\\
&= C \frac{2^{\frac{k+1}{2} A_{p,q}}}{1-2^{\frac{A_{p,q}}{2}}}\left\|f\right\|_{p'},   
\end{align*}
for all functions $f \in \cS_G$, as claimed. 
\end{proof}

\begin{lemma}\label{lem:phi1}
Let $p > \frac{2N}{N-1}\frac{2q}{q +2}$ and suppose that $(v_{n})_{n} \subset \SR_G$ is an $L^{p'}$-bounded sequence with 
$$ 
\lim\limits_{n \to \infty}\: \sup\limits_{y \in \RN} \int\limits_{B_{\rho}(y)}|Q v_{n}|^{p'}~dx  = 0, \quad \text{ for all } \rho > 0. 
$$
Then
$$ \int\limits_{\RN} Q v_{n} [\Phi_{1} \ast (Q v_{n})] ~dx \to 0, \qquad \text{as } n \to \infty. $$
\end{lemma}
\begin{proof}
Let, as in the assumptions of Proposition~\ref{prop:lambdaestimate}, $\varphi \in \SR(\RN)$ be such that $\wdh{\varphi} \in \ceinftyc(\RN)$ is radial, $0\leq \wdh{\varphi} \leq 1$ with $\wdh{\varphi} \equiv 1$ for $\left||\xi| - 1\right| \leq \frac14$ and $\wdh{\varphi} \equiv 0$ for $\left||\xi| - 1\right| \geq \frac12$. Moreover, let $w_{n} = \varphi \ast (Q v_{n})$. Then we have 
\begin{equation}
  \label{eq:young-inequality-varphi-w-n}
\|w_n\|_{p'} \le \|\varphi\|_1 \|Q v_n\|_{p'} \le \|\varphi\|_1 \|Q\|_\infty \|v_n\|_{p'}
\end{equation}
for all $n \in \N$ by Young's inequality, so $(w_n)_n$ is also a bounded sequence in $L^{p'}(\R^N)$ by assumption. Since $\wdh{\Phi_1}= \wdh{\Phi_1}\, \wdh{\varphi}$, we have $(2\pi)^{\frac{N}{2}} \Phi_1 =  \Phi_1 \ast \varphi$. Therefore, with $\eta_k$ defined as in Proposition~\ref{prop:lambdaestimate},  we can write
\begin{align*}
(2\pi)^{\frac{N}{2}} \int\limits_{\RN} Q v_{n} [\Phi_{1} \ast (Qv_{n})]~dx &=  \int\limits_{\RN} Q v_{n} [\Phi_{1} \ast \varphi \ast (Qv_{n})]~dx\\
&= \int\limits_{\RN} Q v_{n} [ \eta_k \Phi_{1} \ast \varphi \ast (Qv_{n})]~dx+
 \int\limits_{\RN} Q v_{n} [ (1-\eta_k) \Phi_{1} \ast \varphi \ast (Qv_{n})]~dx,
\end{align*}
for every $n,k \in \N$, where 
$$
\Bigl| \int\limits_{\RN} Q v_{n} [ (1-\eta_k) \Phi_{1} \ast \varphi \ast (Qv_{n})]~dx\Bigr| \le \|v_n\|_{p'} \left\|Q \Bigl([(1-\eta_k) \Phi_1]\ast  \varphi \ast (Qv_n)\Bigr) \right\|_{p} \leq C \frac{2^{\frac{k+1}{2} A_{p,q}}}{1-2^{\frac{A_{p,q}}{2}}}\left\|v_n\right\|_{p'}^2
$$
by Hölder's inequality and Proposition~\ref{prop:lambdaestimate}. Since $A_{p,q}<0$, it follows that 
\begin{equation}
  \label{prop-lambda-est-consequence}
 \sup\limits_{ n\in \N}  \Bigl| \int\limits_{\RN} Q v_{n} [ (1-\eta_k) \Phi_{1} \ast \varphi \ast (Qv_{n})]~dx \Bigr| \to 0, \qquad \text{as } k \to \infty.
\end{equation}
For fixed $k \in \N$, we now choose $R= 2^{k+1}$, which implies that $\eta_k \equiv 0$ on $\R^N \setminus B_R$. 
Decomposing $\RN$ into disjoint $N$-cubes $\{Z_{l}\}_{l \in \N}$ of side 
length $R$, and considering for each $l$ the $N-$ cube $Z'_{l}$ with the same center as $Z_{l}$ but with side length $3R$, we find
\begin{align*}
&\Bigl| \int\limits_{\RN} Q v_{n} [ \bigl(\eta_k \Phi_{1}\bigr) \ast \varphi \ast (Qv_{n})]~dx \Bigr|
= \Bigl| \int\limits_{\RN} Q v_{n} [ \bigl(\eta_k \Phi_{1}\bigr) \ast w_n]~dx \Bigr|\\
&\leq  \sum\limits^{\infty}_{l=1} \int\limits_{Z_{l}} \Bigl( \int\limits_{|x-y|<R}|\Phi_{1}(x-y)|~|Q v_{n}(x)|~|w_{n}(y)|~dx \Bigr)~dy\\
&\le 
\|\Phi_{1}\|_{L^{\infty}} \sum\limits^{\infty}_{l=1} \int\limits_{Z'_{l}} |Q v_{n}| ~dx\int\limits_{Z'_{l}}|w_{n}|~dx \\
&\leq 
\|\Phi_{1}\|_{L^{\infty}}  \Bigl[\sum\limits^{\infty}_{l=1}\Bigl(\int\limits_{Z'_{l}}|w_{n}|~dx\Bigr)^{p'}\Bigr]^{\frac{1}{p'}}\Bigl[\sum\limits^{\infty}_{l=1}\Bigl( \int\limits_{Z'_{l}}|Q v_{n}|~dx \Bigr)^{p}\Bigr]^{\frac{1}{p}} \\
&\leq
\|\Phi_{1}\|_{L^{\infty}}  (3R)^{\frac{2N}{p}} \Bigl[\sum\limits^{\infty}_{l=1}\int\limits_{Z'_{l}}|w_{n}|^{p'}~dx\Bigr]^{\frac{1}{p'}}\Bigl[\sum\limits^{\infty}_{l=1}\Bigl(\int\limits_{Z'_{l}}|Q v_{n}|^{p'}~dx\Bigr)^{\frac{p}{p'}}\Bigr]^{\frac{1}{p}}\\
&\leq
\|\Phi_{1}\|_{L^{\infty}}   (3R)^{\frac{2N}{p}} 3^{\frac{N}{p'}} \|w_{n}\|_{L^{p'}} \Bigl[\sup\limits_{l \in \N}\int\limits_{Z'_{l}}|Q v_{n}|^{p'}~dx\Bigr]^{\frac{p}{p'}-1} \Bigl[ \sum\limits^{\infty}_{l=1}\int\limits_{Z'_{l}}|Q v_{n}|^{p'}~dx \Bigr]^{\frac{1}{p}}\\
&\leq \|\Phi_{1}\|_{L^{\infty}}   (3R)^{\frac{2N}{p}} 3^{N}\|w_{n}\|_{L^{p'}}
 \Bigl[\sup\limits_{y \in \RN}\int\limits_{B_{3R\sqrt{N}}(y)}|Qv_{n}|^{p'}~dx\Bigr]^{\frac{p}{p'}-1} \|Q v_{n}\|^{\frac{p'}{p}}_{p'} \\
&\leq \|\Phi_{1}\|_{L^{\infty}} \|Q\|_\infty^{1+\frac{p'}{p}} \|\varphi\|_1  (3R)^{\frac{2N}{p}} 3^{N}
\|v_{n} \|^{1+\frac{p'}{p}}_{p'}\Bigl[\:\sup\limits_{y \in \RN}\int\limits_{B_{3R\sqrt{N}}(y)}|Q v_{n}|^{p'}~dx\Bigr]^{\frac{p}{p'}-1}, 
\end{align*}
where we used (\ref{eq:young-inequality-varphi-w-n}) in the last step. By assumption, it now follows that 
\begin{equation}
  \label{R-tiling-estimate}
\int\limits_{\RN} Q v_{n} [ \bigl(\eta_k \Phi_{1}\bigr) \ast \varphi \ast (Qv_{n})]~dx \to 0 \qquad \text{as $n \to \infty$ for every $k \in \N$.}
\end{equation}
The claim now follows by combining (\ref{prop-lambda-est-consequence}) and (\ref{R-tiling-estimate}). 
\end{proof}	
\noindent

Regarding $\Phi_{2}$ we make use of the following variant of \cite[Theorem 2.5]{Evequoz2019}. 
\begin{lemma}\label{lem:phi2} 
Let $2 < p \leq \frac{2N}{N-2}$ and suppose that $(v_n)_n$ is a bounded sequence in $L^{p'}(\RN)$ such that
$$ \lim\limits_{n \to \infty} \sup\limits_{y \in \RN}\int\limits_{B_{\rho}(y)}|v_{n}|^{p'}~dx = 0, \qquad \text{ for all } \rho > 0. $$
Then 
$$ \int\limits_{\RN}v_{n} [\Phi_{2}\ast v_{n}] ~dx \to 0 \quad \text{ as } n \to \infty.$$
\end{lemma}

\begin{proof}
The claim follows from \cite[Theorem 2.5]{Evequoz2019} in the case where $v_{n} \in \SR$ for every $n \in \N$.
If $(v_n)_n$ is an arbitrary bounded sequence in $L^{p'}(\R^N)$, we first recall that, by Lemma~\ref{R-2-estimate}, there exists a constant $C>0$ with 
$$
\|\Phi_2 \ast v \|_p \le C \|v\|_{p'} \qquad \text{for every $v \in L^{p'}(\R^N)$.}
$$
Moreover we choose, by density, $\tilde v_n \in \cS$ with 
$\|v_n-\tilde v_n\|_{p'} \le \frac{1}{n}$ for every $n \in
\N$. The assumption then implies that also 
$$ 
\lim\limits_{n \to \infty} \sup\limits_{y \in \RN}\int\limits_{B_{\rho}(y)}|\tilde v_{n}|^{p'}~dx  = 0, \qquad \text{ for all } \rho > 0 
$$
and therefore 
$$ 
\int\limits_{\RN} \tilde v_{n} [\Phi_{2}\ast \tilde v_{n}] ~dx \to 0, \quad \text{ as } n \to \infty
$$
by \cite[Theorem 2.5]{Evequoz2019}. Moreover,
\begin{align*}
&\Bigl|\int_{\R^N}[v_n (\Phi_{2}\ast v_{n})  -\tilde v_n (\Phi_{2}\ast \tilde v_{n}) ]\,
  dx\Bigr | = \Bigl|\int_{\R^N} (v_n-\tilde v_n) \Phi_2 \ast (v_n+\tilde
v_n)\,dx \Bigr|\\
&\le C \|v_n-\tilde v_n\|_{p'} \|v_n+\tilde
v_n\|_{p'}
\le \frac{C(1+\frac{1}{n})\|v_n\|_{p'}}{n} 
 \to 0 \qquad \text{as $n \to \infty$}
\end{align*}
and thus also 
$$ 
\int\limits_{\RN}v_{n} [\Phi_{2}\ast v_{n}] ~dx \to 0 \quad \text{ as } n \to \infty,
$$
as claimed. 
\end{proof}

We are now in position to finish the proof of Theorem~\ref{nonvanishing-section}:

\begin{proof}[Proof of Theorem~\ref{nonvanishing-section}]
Let $(v_{n})_{n} \subset L^{p'}_G(\RN)$ be a bounded sequence, and suppose by contradiction that \eqref{eqn:nonvanishingsymmetric} does not hold. Then we have
	$$
\lim\limits_{n\to \infty} \sup\limits_{y \in \RN} \int\limits_{B_{\rho}(y)} |Q v_{n}|^{p'}~dx  = 0, \qquad \text{ for all } \rho > 0.
$$
By density, we may choose $\tilde v_n \in \cS_G$ with 
$\|v_n-\tilde v_n\|_{p'} \le \frac{1}{n}$ for every $n \in
\N$, which implies that $\|Qv_n-Q\tilde v_n\|_{p'} \le \frac{\|Q\|_\infty}{n}$ for all $n$ and therefore also  
$$
\lim\limits_{n\to \infty} \sup\limits_{y \in \RN} \int\limits_{B_{\rho}(y)} |Q \tilde v_{n}|^{p'}~dx  = 0, \qquad \text{ for all } \rho > 0.
$$
Combining Lemma \ref{lem:phi1} (applied to $\tilde v_n$) and Lemma \ref{lem:phi2} (applied to $Q \tilde v_n$), we then deduce that 
	$$ 
\int\limits_{\RN} \tilde v_{n} \cR_\sQ \tilde v_{n} dx = \int\limits_{\RN} Q \tilde v_{n} [\Phi_{1} \ast (Q \tilde v_{n})]~dx + \int\limits_{\RN} Q \tilde v_{n}[\Phi_{2}\ast (Q \tilde v_{n})]~dx \rightarrow 0 \qquad \text{ as } n \to \infty. 
$$ 
Moreover, by Theorem~\ref{resolvent-estimates-intro} we have 
\begin{align*}
&\Bigl|\int_{\R^N}[v_n \cR_\sQ v_n  -\tilde v_n \cR_\sQ \tilde v_{n} ]\,
  dx\Bigr | = \Bigl|\int_{\R^N} (v_n-\tilde v_n) \cR_\sQ (v_n+\tilde
v_n)\,dx \Bigr|\\
&\le \|v_n-\tilde v_n\|_{p'} \|\cR_\sQ(v_n+\tilde
v_n)\|_{p} \le C \|v_n-\tilde v_n\|_{p'} \|v_n+\tilde
v_n\|_{p}\\
&\to 0 \qquad \text{as $n \to \infty$.}
\end{align*}
Consequently, we also have that $\int\limits_{\RN} v_{n} \cR_\sQ v_{n} dx \to 0$ as $n \to \infty$, contrary to the assumption. The claim thus follows.
\end{proof}

\section{Dual variational framework and $G-$invariant solutions}
\label{sec:dual-vari-fram}

Let $G \subset O(N)$ be a fixed closed subgroup, and let $Q \in L^\infty_G(\R^N)$ be a nonnegative fixed weight function with $Q \not \equiv 0$. We now focus our attention to the equation
\begin{equation}\label{eq}
-\Delta u - u = Q(x)|u|^{p-2}u, \qquad u \in L^{p}(\RN).
\end{equation}
To prove the existence of nontrivial real-valued solutions of (\ref{eq}), we will  use the dual variational approach introduced in \cite{Evequoz2015} and consider the operator $K_\sQ$ formally defined as $K_\sQ f = Q^{\frac{1}{p}} R (Q^{\frac{1}{p}}f)$, where 
$R$ denotes the real part of the Helmholtz resolvent operator $\cR$, i.e., $R g = \bigl({\rm Re\,} \Phi\bigr)\ast g$ with the fundamental solution $\Phi$ defined in (\ref{fundsol}).

To analyze the mapping properties of $K_\sQ$ and to set up a variational framework, we assume, as in Theorem~\ref{dual-ground-state-intro}, that $q \in \bigl[1,\frac{2(N+1)}{N-1}\bigr]$ and $p \in \bigl(\max\{\frac{2N}{N-1} \frac{2 q}{q+ 2},2\}, \frac{2N}{N-2}\bigr)$ are fixed such that $(G,q,Q^{\frac{1}{p}})$ is an admissible extension triple. 
The following is an immediate consequence of Theorem~\ref{resolvent-estimates-non-self-dual}. 
\begin{lemma}
\label{non-selfdual-0}
Let $\tilde p,r \in (1,\infty)$ satisfy (\ref{eq:p-r-condition-non-self-dual-general}) with $\tilde p$ in place of $p$, and suppose moreover that (\ref{eq:p-r-condition-non-self-dual-q-greater-2}) holds with $\tilde p$ in place of $p$ if $q \ge 2$, and that (\ref{eq:p-r-condition-non-self-dual-q-smaller-2}) holds with $\tilde p$ in place of $p$ if $q<2$.

Then the operator $K_\sQ$ is bounded as a map $L_G^{\tilde p'}(\R^N) \to L_G^r(\R^N)$.
\end{lemma}

We note that Lemma~\ref{non-selfdual-0} applies in particular in the case $r=\tilde p=p$, so
$$
\text{$K_\sQ$ is a bounded operator $L_G^{p'}(\R^N) \to L_G^p(\R^N)$.}
$$
We also note the following immediate corollary of Lemma~\ref{non-selfdual-0}.

\begin{corr}
\label{non-selfdual}
There exist $\sigma_1 < p<\sigma_2$ with the property that $K_\sQ$ is bounded as a map $L_G^{\sigma_i'}(\R^N) \to L_G^p(\R^N)$ and as a map 
$L_G^{p'}(\R^N) \to L_G^{\sigma_i}(\R^N)$ for $i=1,2$. 
\end{corr}

Next we note the following variant of \cite[Lemma 4.1]{Evequoz2015}.

\begin{lemma}
\label{K-locally-compact}
The operator $K_{\sQ}: L_G^{p'}(\RN) \to L_G^{p}(\RN)$ is locally compact, i.e., the operators 
$$
K_{\sQ}\mathds{1}_{B} : L_G^{p'}(\RN) \to L_G^{p}(\RN) \qquad \text{and}\qquad \mathds{1}_{B}K_{\sQ} : L_G^{p'}(\RN) \to L_G^{p}(\RN)
$$
are compact for every bounded and measurable set $B \subset \RN$. 
\end{lemma}

\begin{proof}
Let $B \subset \RN$ be bounded and measurable, and fix $s \in \bigl[\frac{2(N+1)}{N-1}, \frac{2N}{N-2}\bigr)$ with $s \ge p$, i.e., $s' \le p'$.  By \cite[Lemma 4.1]{Evequoz2015}, the operator $\mathds{1}_{B}K_{\sQ} : L^{s'}(\RN) \to L^{s}(\RN)$ is compact. By duality, the operator 
$K_{\sQ}\mathds{1}_{B} : L^{s'}(\RN) \to L^{s}(\RN)$ is therefore also compact.

Next, let $(v_n)_n \subset L_G^{p'}(\RN)$ be a sequence with $v_n \rightharpoonup 0$ in $L_G^{p'}(\RN)$. Then we have 
$w_n := \mathds{1}_B v_n \rightharpoonup 0$ in $L_G^{p'}(\R^N)$, and thus also in $L_G^{s'}(\R^N)$, since $B$ has finite measure. 
By the compactness property mentioned above, it follows that $K_{\sQ}\mathds{1}_{B}v_n = K_{\sQ}\mathds{1}_{B} w_n \to 0$ strongly in $L_G^{s}(\RN)$. Moreover, it follows from Corollary~\ref{non-selfdual} that the sequence of functions $K_{\sQ}\mathds{1}_{B}v_n = K_\sQ w_n$, $n \in \N$ is bounded in $L_G^{\sigma_{1}}(\R^N)$ for some $\sigma_1<p$. Since $\sigma_1<p \le s$, it thus follows by interpolation that there exists $\theta \in (0,1]$ with 
$$
\|K_{\sQ}\mathds{1}_{B}v_n\|_p \le \|K_{\sQ}\mathds{1}_{B}v_n\|^{1-\theta}_{\sigma_{1}} \|K_{\sQ}\mathds{1}_{B}v_n\|_s^\theta \to 0 
\qquad \text{as $n \to \infty$.} 
$$
Hence the operator $K_{\sQ}\mathds{1}_{B} : L_G^{p'}(\RN) \to L_G^{p}(\RN)$ is compact, and by duality it follows that also 
$\mathds{1}_{B}K_{\sQ} : L_G^{p'}(\RN) \to L_G^{p}(\RN)$ is a compact operator.
\end{proof}

As in \cite{Evequoz2015}, we now introduce the (dual) energy functional
$$
J : L_G^{p'}(\RN) \to \R, \qquad J(v) = \frac{1}{p'}\int\limits_{\RN}|v|^{p'}~dx - \frac{1}{2}\int\limits_{\RN}v[K_{\sQ}v](x)~dx. 
$$
Then $J$ is of class $\cC^1$ with 
$$ 
J'(v)w = \int\limits_{\RN}(|v|^{p'-2}v -K_{\sQ}v)w~dx \qquad \text{for $v,w \in L_G^{p'}(\R^N)$.} 
$$
Moreover, we have 
\begin{lemma}
\label{crit-point-solution}
If $v \in L_G^{p'}(\R^N)$ is a critical point of $J$, then $u = R Q^{\frac{1}{p}}v$ is a real-valued solution of 
(\ref{eq}) of class $W^{2,q}(\RN) \cap \cC^{1,s}(\RN)$ for $q \ge p$, $s \in (0,1)$.
\end{lemma}

\begin{proof}
Let $w \in L^{p'}(\RN)$, and let $w_G \in L^{p'}(\R^N)$ be defined by 
$$ 
w_{G}= \int\limits_{G} w \circ A d\mu(A),\qquad \text{i.e.,}\qquad  w_{G}(x) = \int\limits_{G} w(Ax) d\mu(A) \quad \text{for $x \in \R^N$,}
$$
where $\mu$ is the Haar-measure of $G$. Since $v$ is $G$-invariant, it follows that 
$$ 
\int\limits_{\RN}(|v|^{p'-2}v -K_{\sQ}v)w~dx = \int\limits_{\RN}(|v|^{p'-2}v -K_{\sQ}v)[w \circ A]~dx \qquad \text{for all $A \in G$}
$$
and therefore 
$$
\int\limits_{\RN}(|v|^{p'-2}v -K_{\sQ}v)w~dx = \int\limits_{\RN}(|v|^{p'-2}v -K_{\sQ}v)w_G~dx = J'(v)w_G = 0
$$
Consequently, we have $|v|^{p'-2}v = K_{\sQ}v$ in $L^{p}(\R^N)$, which implies that $u = R Q^{\frac{1}{p}}v$ satisfies the equation
$$
u = R Q|u|^{p-2}u \qquad \text{in $L^p(\R^N)$.}
$$ 
The claim now follows by \cite[Lemma 4.3]{Evequoz2015}.
\end{proof}

Next we note that the functional $J$ has a mountain pass geometry. More precisely, we have:

\begin{lemma}\label{lem:MP_geom2}
\begin{itemize}
	\item[(i)] There exists $\delta>0$ and $0<\rho<1$ such that $J(v)\geq\delta>0$ for all $v\in L^{p'}_G(\R^N)$ with $\|v\|_{p'}=\rho$.
	\item[(ii)] There is $v_0\in L^{p'}_G(\R^N)$ such that $\|v_0\|_{p'}>1$ and $J(v_0)<0$.
	\item[(iii)] Every Palais-Smale sequence for $J$ is bounded in $L^{p'}_G(\R^N)$.
\item[(iv)] There exists a Palais-Smale sequence for $J$ at the mountain pass level 
  \begin{equation}
    \label{eq:def-mountain-pass-level}
d:= \inf_{\gamma \in \Gamma}\max_{t \in [0,1]}J(\gamma(t)) \:>\:0,
  \end{equation}
where $\Gamma= \{\gamma \in C([0,1],L^{p'}_G(\R^N))\::\: \gamma(0)=0,\, \|\gamma(1)\|_{p'}>\rho,\, J(\gamma(1))<0\}$.
\end{itemize}
\end{lemma}

\begin{proof}
Since $p>2$, the parts (i)-(iii) are proved in \cite[Lemma 4.2]{Evequoz2015} for $G=\{id\}$, and the proof remains the same for general closed subgroups $G \subset O(N)$. 
Moreover, the positivity of the mountain pass level $c$ defined in (\ref{eq:def-mountain-pass-level}) is a direct consequence of (i) and (ii), which also shows that the set $\Gamma$ is nonempty. Finally, the proof of the existence of a Palais-Smale sequence for $J$ at level $d$ is exactly the same as the proof of \cite[Lemma 6.1]{Evequoz2015}. Here we note that periodicity of $Q$ was assumed in \cite[Section 6]{Evequoz2015}, but this property is not used in Lemma 6.1.  
\end{proof}

\begin{proposition}
\label{nontrivial-weak-limit-PS}
Let $(v_n)_n \subset L^{p'}_G(\R^N)$ be a Palais-Smale sequence of $J$ with $c:= \lim \limits_{n \to \infty}J(v_n) >0$. Moreover, suppose that {\em one} of the following conditions hold: 
\begin{itemize}
\item[(A1)] For some $R>0$, we have $\lim \limits_{|x| \to \infty} \|Q\|_{L^1(B_R(x))}= 0$.
\item[(A2)] For every $R>0$ we have $\lim \limits_{|x| \to \infty} N_G(x,R)= \infty$, where, for $R>0$ and $x \in \R^N \setminus \{0\}$, $N_G(x,R)$ denotes the maximal number of elements of a subset $H \subset G$ with $B_{R}(Ax) \cap B_{R}(A'x)= \varnothing$ for $A, A' \in H$.
\end{itemize}
Then, after passing to a subsequence, we have 
$$
v_n \rightharpoonup v  \qquad \text{in $L^{p'}(\R^N)$,}  
$$
where $v \in L^{p'}_G(\R^N) \setminus \{0\}$ is a critical point of $J$. 
\end{proposition}

\begin{proof}
We first note that $(v_n)_n$ is bounded by Lemma~\ref{lem:MP_geom2}. Consequently, since $L^{p'}_G(\R^N)$ is reflexive, there exists $v \in L^{p'}_G(\R^N)$ such that 
\begin{equation}
  \label{eq:v-n-weak-conv}
v_n \rightharpoonup v  \qquad \text{in $L^{p'}(\R^N)$.}  
\end{equation}
Moreover, 
$$ 
\lim\limits_{n\to\infty}\int\limits_{\RN} v_{n}K_\sQ v_{n}\,dx = \frac{2p'}{2-p'} \lim\limits_{n\to\infty}\left( J(v_{n}) - \frac{1}{p'}J'(v_{n})v_{n}\right) = \frac{2p'}{2-p'} c>0 
$$
by assumption, which implies that 
$$ 
\lim\limits_{n\to\infty} \Bigl|\int\limits_{\RN} v_{n} \cR_{\text{\tiny $Q^{\frac{1}{p}}$}} v_{n}~dx\Bigr| >0. 
$$
Since moreover $(G,q,Q^{\frac{1}{p}})$ is an admissible extension triple by assumption, Theorem~\ref{nonvanishing-intro} applies and yields $\delta, R >0$ and a sequence of points $(x_{n})_{n} \subset \RN$ such that, after passing to a subsequence,
\begin{equation}\label{nonvanish}
\int\limits_{B_{R}(x_{n})}|Q^{\frac{1}{p}}v_{n}|^{p'}~dx \geq  \delta > 0, \quad \text{for all } n \in \N.
\end{equation} 
We claim that $(x_{n})_{n}$ has to be bounded. To see this, we argue by contradiction and assume that, after passing to a subsequence again, $|x_n| \to \infty$.
We distinguish two cases.\\
{\bf Case 1: (A1) holds.}\\
In this case we put $\varphi_n := Q^{p-1}v_n \mathds{1}_{B_R(x_n)}$, and we note that $(\varphi_n)_n$ is a bounded sequence in $L^{p'}(\R^N)$. Moreover, we have 
\begin{align}
  \label{eq:nonvanish-contradiction}
\int\limits_{B_{R}(x_{n})}|Q^{\frac{1}{p}}v_{n}|^{p'}~dx&= \int\limits_{\RN}|v_{n}|^{p'-2}v_{n}\varphi_ndx
= J'(v_{n})\varphi_{n}  + \int\limits_{\RN}v_{n}K_\sQ \varphi_{n}dx\\ 
&\leq o(1)\left\|\varphi_n\right\|_{p} + \Bigl| \int\limits_{\RN} v_{n}K_\sQ \varphi_{n} ~dx\Bigl| = o(1) + \|v_{n}\|_{p'} \|K_\sQ \varphi_{n}\|_{p}
\nonumber
\end{align}
as $n \to \infty$. By Corollary~\ref{non-selfdual}, there exists $\sigma>p'$ and $C>0$ with the property that 
$$
 \|K_\sQ \varphi_{n}\|_{p}\leq C  \|\varphi_{n}\|_{\sigma'}\qquad \text{for $n \in \N$,} 
$$
whereas, since $\sigma' < p'$ and by Hölder's inequality,
\begin{align*}
\|\varphi_{n}\|_{\sigma'} &= \|Q^{p-1}v_n\|_{L^{\sigma'}(B_R(x_n))}
\le  \Bigl(\int_{B_R(x_n)} |Q|^{\frac{p' \sigma'(p-1)}{p'-\sigma'}}dx \Bigr)^{\frac{p'-\sigma'}{p' \sigma'}} \|v_{n}\|_{p'}\\
&\le \Bigl(\|Q\|_{L^1(B_R(x_n))}\Bigr)^{\frac{p'-\sigma'}{p' \sigma'}} 
\|Q\|_\infty^{\bigl(\frac{p' \sigma'(p-1)}{p'-\sigma'}-1\bigr)\frac{p'-\sigma'}{p' \sigma'}} \|v_{n}\|_{p'}
\end{align*}
Since $\|Q\|_{L^1(B_R(x_n))} \to 0$ by (\ref{eq:asymptotic-Q-condition}), it thus follows that $\|\varphi_n\|_{\sigma'} \to  0$ as $n \to \infty$.
Here we note that, by an easy covering argument, (\ref{eq:asymptotic-Q-condition}) holds for every $R>0$ if it holds for one $R>0$. Going back to (\ref{eq:nonvanish-contradiction}), we thus deduce that 
$$
\int\limits_{B_{R}(x_{n})}|Q^{\frac{1}{p}}v_{n}|^{p'}~dx \to 0 \qquad \text{as $n \to \infty$,}
$$
which contradicts (\ref{nonvanish}).\\ 
{\bf Case 2: (A2) holds.}\\
In this case it follows from (\ref{nonvanish}) and the fact that $v_n$ and $Q$ are $G$-invariant that 
$$
\|Q^{\frac{1}{p}} v_n\|_{p'}^{p'} \ge N_G(x_n,R)  \int\limits_{B_{R}(x_{n})}|Q^{\frac{1}{p}}v_{n}|^{p'}~dx \geq N_G(x_n,R) \delta \to \infty
$$
as $n \to \infty$, which contradicts the boundedness of the sequence $(v_n)_n$
in $L^{p'}(\R^N)$.\\
Since in both cases we have reached a contradiction, we conclude that $(x_{n})_{n}$ is bounded. Therefore, making $R$ larger if necessary, we can assume that \eqref{nonvanish} holds with $x_{n} = 0$ for all $n \in \N$.  
Now for any  fixed $G$-invariant function $\varphi \in \ceinftyc(\RN)$, any $r> 0$ and $n,m \in \N$ we have
\begin{align*}
	\left| \int\limits_{\RN} \left(|v_{n}|^{p'-2}v_{n} - |v_{m}|^{p'-2}v_{m}\right)\varphi~  dx \right| &=\left|J'(v_{n})\varphi - J'(v_{m})\varphi + \int\limits_{B_{r}}\varphi K_\sQ (v_{n}- v)~dx \right|\\
&\leq
\left\|J'(v_{n}) - J'(v_{m})\right\|\left\|\varphi\right\|_{p'} + \left\|\mathds{1}_{B_{r}}K_\sQ(v_{n} - v_{m})\right\|_{p} \left\|\varphi\right\|_{p'}.
\end{align*}
So by assumption and the local compactness of $K_\sQ$, as stated in Lemma~\ref{K-locally-compact}, we get that $(|v_{n}|^{p'-2}v_{n})_{n \in \N}$ is a Cauchy sequence in $L^{p}(B_{R})$. Consequently, $|v_{n}|^{p'-2}v_{n} \to \tilde{v}$ strongly in $L^{p}(B_{R})$ for some $\tilde{v} \in L^{p}(B_{R})$, and passing to a subsequence also pointwisely almost everywhere on $B_{R}$. This clearly implies that $v_{n} \to |\tilde{v}|^{p-2}\tilde{v}$ almost everywhere on $B_r$. Now (\ref{eq:v-n-weak-conv}) and the uniqueness of the weak limit gives $\tilde{v} = |v|^{p'-2}v$ and 
$$ 
0<\delta \leq \int\limits_{B_{R}}|Q^{\frac{1}{p}}v_{n}(x)|^{p'}~dx \to \int\limits_{B_{R}}|Q^{\frac{1}{p}}v|^{p'}~dx, \quad \text{as } n \to \infty 
$$
which implies that $v \neq 0$.\\
For every $G$-invariant function $\varphi \in\ceinftyc$, we now have
\begin{align*}
J'(v)\varphi &= \int\limits_{\RN} |v|^{p'-2}v \varphi ~dx - \int\limits_{\RN} \varphi K_\sQ(v)~dx \\
&= \lim\limits_{n\to \infty} \left[\int\limits_{\RN} |v_{n}|^{p'-2}v_{n} \varphi ~dx - \int\limits_{\RN} \varphi K_\sQ(v_{n})~dx\right] \\
&= \lim\limits_{n\to\infty}J'(v_{n})\varphi  = 0
\end{align*}
using the local strong convergence of $|v_{n}|^{p'-2}v_{n}$ and the continuity of linear operator $K_\sQ: L^{p'}_G(\R^N) \to L^{p}_G(\R^N)$. By density, it now follows that $J'(v)w= 0$ for every $w \in L^{p'}_G(\R^N)$, i.e., $v \in L^{p'}_G(\RN)\setminus\{0\}$ is a critical point of $J$.
\end{proof}

We now have all the tools to complete the proofs of our main existence results for nontrivial $G$-invariant dual ground state solutions as stated in the introduction. 

\begin{proof}[Proof of Theorem~\ref{dual-ground-state-intro} (completed)]
By Lemma~\ref{lem:MP_geom2}(iv), there exists a Palais-Smale sequence $(v_n)_n$ in $L^{p'}_G(\R^N)$ for $J$ at the mountain pass level $d>0$. By Proposition~\ref{nontrivial-weak-limit-PS}, we have $v_n \rightharpoonup v$ in $L^{p'}_G(\R^N)$ after passing to a subsequence, where $v \in L^{p'}_G(\R^N)$ is a nontrivial critical point of $J$. Here we note that assumption (A1) of Proposition~\ref{nontrivial-weak-limit-PS} is satisfied by (\ref{eq:asymptotic-Q-condition}).   
The proof is finished by Lemma~\ref{crit-point-solution}.
\end{proof}

\begin{proof}[Proof of Corollary~\ref{dual-ground-state-intro-cor-2}]
Since $Q^{\frac{1}{p}} \in L^\infty(\R^N)$, it follows by the classical Stein-Tomas estimate that $(G,q,Q^{\frac{1}{p}})$ is an admissible extension triple for $q= 
\frac{2(N+1)}{N-1}$. Since 
$$
p \in \bigl(\frac{2(N+1)}{N-1},\frac{2N}{N-2}\bigr) = \bigl(\frac{2N}{N-1} \frac{2 q}{q+ 2}, \frac{2N}{N-2}\bigr),
$$
the assumptions of Theorem~\ref{dual-ground-state-intro} are satisfied and yield the existence of a nontrivial solution $v \in L^{p'}_G(\R^N)$ of (\ref{eqn:intro_helmholtz}). 
\end{proof}

\begin{proof}[Proof of Theorem~\ref{dual-ground-state-intro-cor-3}]
As above, it follows by the classical Stein-Tomas estimate that $(G_k,q,Q^{\frac{1}{p}})$ is an admissible extension triple for $q= 
\frac{2(N+1)}{N-1}$, whereas   
$$
p \in \bigl(\frac{2(N+1)}{N-1},\frac{2N}{N-2}\bigr) = \bigl(\frac{2N}{N-1} \frac{2 q}{q+ 2}, \frac{2N}{N-2}\bigr).
$$
Moreover, since $2 \le k \le N-2$, we have
\begin{equation}
  \label{eq:N_G-infty-application}
\lim \limits_{|x| \to \infty} N_G(x,R)= \infty \qquad \text{for every $R>0$,}  
\end{equation}
where $N_G(x,R)$ is defined as in Proposition~\ref{nontrivial-weak-limit-PS}. This fact is noted without proof in \cite[Proof of Corollary 1.25]{Willem:1996}, and we give the short proof here for the reader's convenience. In fact, (\ref{eq:N_G-infty-application}) follows already from the fact that
the minimal orbit dimension of $G_k$ is $\min \{k-1,N-k-1\}$ and therefore greater than or equal to one by assumption. In particular, for every $n \in \N$ and $\theta \in \SPN$, there exists $\eps>0$ and a subset $H_\theta \subset G$ with $B_{\eps}(A \theta) \cap B_{\eps}(A'\theta)= \varnothing$ for every $A, A' \in H_\theta$. Moreover, by a straightforward compactness argument, $\eps>0$ can be chosen to depend only on $n$ and not on $\theta \in \SPN$. Hence, if $R>0$ is given, $x \in \R^N$ satisfies $|x| \ge \frac{R}{\eps}$ and $\theta$ equals $\frac{x}{|x|}$, we have $B_{R}(A \theta) \cap B_{R}(A'\theta)= \varnothing$ for every $A, A' \in H_\theta$ and therefore $N_G(x,R) \ge n$. This shows (\ref{eq:N_G-infty-application}).

Hence assumption (A2) of Proposition~\ref{nontrivial-weak-limit-PS} is satisfied, and thus the proof is completed as the proof of Theorem~\ref{dual-ground-state-intro} above.
\end{proof}

\begin{proof}[Proof of Corollary~\ref{dual-ground-state-intro-cor-1}]
  We first note that $Q$ satisfies the asymptotic condition (\ref{eq:asymptotic-Q-condition}). Indeed, since $0 \le Q \le c \mathds{1}_{L_\alpha}$ for some $c>0$ by assumption, it suffices to show that 
\begin{equation}
\label{eq:asymptotic-Q-condition-special-case}
|L_\alpha \cap B_{R}(x) | \to 0 \qquad \text{as $|x| \to \infty$ for every $R>0$.}  
\end{equation}
To see the latter, it suffices to consider a sequence $(x_{n})_{n} = (x^{(N-k)}_{n},x^{(k)}_{n}) \subset \R^{N-k} \times \R^{k}$ with $x^{(N-k)}_{n} = 0$ for all $n \in \N$ and $r_n:= |x_n| = |x^{(k)}_{n}|  \to \infty$ as $n \to \infty$. In this case we have $|x^{(k)} - x^{(k)}_{n} | < R$ for $x \in B_{R}(x_{n})$
and therefore 
\begin{align*}
|L_\alpha \cap B_{R}(x_{n})| &\le \int\limits_{\bigl \{|x^{(k)} - x^{(k)}_{n} | < R\bigr\}} \:\int\limits_{\bigl\{|x^{(N-k)}| \leq a|x^{(k)}|^{-\alpha}\bigr\}} ~dx^{(N-k)}~dx^{(k)} \leq C \int\limits_{\{|x^{(k)} - x^{(k)}_{n} | < R\}} |x^{(k)}|^{-(N-k)\alpha}~dx^{(k)}\\
&= C \int\limits_{|z^{(k)}| < R} |z^{(k)} +x^{(k)}_{n}|^{-(N-k)\alpha} ~dx^{(k)} 
\leq 
C \int\limits_{|z^{(k)}| < R} \left(|x^{(k)}_{n}|  - R \right)^{-(N-k)\alpha}~dx^{(k)} \\
&=C(r_n-R)^{-(N-k)\alpha} \: \to \: 0 \qquad \text{as $n \to \infty$}
\end{align*}  
with constants $C>0$. Hence (\ref{eq:asymptotic-Q-condition-special-case}) holds. 

Next, we first consider the case $k=1$. By Theorem \ref{theo:intro1-alpha-equals-beta} additionally the condition $\alpha > \frac{1}{N-1}$ is required and we set $\lambda = \frac{2(N-1) - \frac{2}{\alpha}}{N-2}$. By case distinction we see  that $\mu_{N,1,\alpha} = \max\left\lbrace \frac{2N}{N-1}\frac{2\lambda}{\lambda+2},2 \right\rbrace$ for  $\alpha > \frac{1}{N-1}$. Thus by Theorem \ref{theo:intro1-alpha-equals-beta} we may fix any $q \in \left(\mu_{N,1,\alpha},p\right)$ with $p \in \left(\mu_{N,1,\alpha}, \frac{2N}{N-2}\right)$ such that $(G_{1},q,\mathds{1}_{L_{\alpha}})$ is an admissible extension triple. Since $0 \leq Q^{\frac{1}{p}} \leq c^{\frac{1}{p}}\mathds{1}_{L_{\alpha}}$, it follows that also $(G_{1},q,Q^{\frac{1}{p}})$ is an admissible extension triple. Thus Theorem \ref{dual-ground-state-intro} applies and yields that \eqref{eqn:intro_helmholtz} admits a dual bound state solution. \\
The case $k=N-1$ now follows similarly: Consider additionally $\alpha < N-1$, set $\lambda = \frac{2(N-1) - 2\alpha}{N-2}$ and observe that for $\alpha < N-1$ the expression $\mu_{N,N-1,\alpha}$ is chosen such that $\mu_{N,N-1,\alpha} = \max\left\lbrace \frac{2N}{N-1}\frac{2\lambda}{\lambda+2},2 \right\rbrace$. Then, for $q,p$ as above with $\mu_{N,N-1,\alpha}$ instead of $\mu_{N,1,\alpha}$ we conclude that $(G_{N-1},q,Q^{\frac{1}{p}})$ is admissible extension triple and Theorem \ref{dual-ground-state-intro} again yields the existence of a dual bound state solution of \eqref{eqn:intro_helmholtz}. \\
If $2 \leq k \leq N-2$ and $p \in \left(\mu_{N,k,\alpha},\frac{2N}{N-2}\right)$  again a case distinction shows that $\mu_{N,k,\alpha} = \max \{ \frac{2N}{N-1} \, \frac{2 \lambda}{\lambda+2}, 2\}$, where $\lambda:=\lambda_{N,k,\alpha}$ is given in Theorem~\ref{theo:intro1-alpha-equals-beta}.  Consequently, by Theorem~\ref{theo:intro1-alpha-equals-beta}, we may fix $q \in (\mu_{N,k,\alpha},p)$ with $\max\{\frac{2N}{N-1} \, \frac{2 q}{q+2}, 2\}< p< \frac{2N}{N-2}$ and 
the property that $(G_k,q,\mathds{1}_{L_\alpha})$ is an admissible extension triple. As above, it follows that also $(G_k,q,Q^{\frac{1}{p}})$ is an admissible extension triple. Again, Theorem~\ref{dual-ground-state-intro} applies and yields that (\ref{eqn:intro_helmholtz}) admits a nontrivial dual bound state solution. Thus the claim holds in this case as well.
\end{proof}

\begin{remark}
We note that Corollary \ref{dual-ground-state-intro-cor-1} extends to the case where $L_\alpha$ is replaced by the more general class of sets $L_{\alpha,\beta}$ considered in Theorem~\ref{theo:alpha-diff-beta-section}. For this, one has to additionally assume $\beta> \frac{1}{N-1}$ if $k=1$.  Then the statement of Corollary \ref{dual-ground-state-intro-cor-1} holds with $\mu_{N,1,\beta}$. If $k = N-1$ the statement holds with the same value $\mu_{N,N-1,\alpha}$. For $2 \leq k \leq N-2$ the value $\mu_{N,k,\alpha}$ needs to be replaced by $\max\left\lbrace\frac{2N}{N-1}\frac{2\lambda}{\lambda + 2},2\right\rbrace$, where now $\lambda = \lambda_{N,k,\alpha,\beta}$ is given in (\ref{eq:def-lambda-N-k-alpha-beta})
\end{remark}

\end{document}